\documentclass[reqno,11pt]{amsart}
\usepackage{newtxtext}
\usepackage{float,amssymb,bm,url}
\usepackage{tikz}
\usepackage{adjustbox}

\usetikzlibrary{cd,calc,positioning,angles,decorations.markings}
\tikzcdset{arrow style=tikz, diagrams={>=stealth}}

\tikzset{
sedge/.append style={shorten <=9pt, shorten >=9pt}
}
\makeatletter
\tikzset{
  @arc through/.style 2 args={
    to path={
      \pgfextra
        \pgfextract@process\pgf@tostart{\tikz@scan@one@point\pgfutil@firstofone(\tikztostart)\relax}%
        \pgfextract@process\pgf@tothrough{\tikz@scan@one@point\pgfutil@firstofone#1}%
        \pgfextract@process\pgf@totarget{\tikz@scan@one@point\pgfutil@firstofone(\tikztotarget)\relax}%
        \pgfextract@process\pgf@topointMidA{\pgfpointlineattime{.5}{\pgf@tostart}{\pgf@tothrough}}%
        \pgfextract@process\pgf@topointMidB{\pgfpointlineattime{.5}{\pgf@totarget}{\pgf@tothrough}}%
        \pgfextract@process\pgf@tocenter{%
          \pgfpointintersectionoflines{\pgf@topointMidA}
            {\pgfmathrotatepointaround{\pgf@tothrough}{\pgf@topointMidA}{90}}
            {\pgf@topointMidB}{\pgfmathrotatepointaround{\pgf@tothrough}{\pgf@topointMidB}{90}}}%
        \pgfcoordinate{arc through center}{\pgf@tocenter}%
        \pgfpointdiff{\pgf@tocenter}{\pgf@tostart}%
        \pgfmathveclen@{\pgfmath@tonumber\pgf@x}{\pgfmath@tonumber\pgf@y}%
        \edef\pgf@toradius{\pgfmathresult pt}
        \pgfmathanglebetweenpoints{\pgf@tocenter}{\pgf@tostart}%
        \let\pgf@tostartangle\pgfmathresult
        \pgfmathanglebetweenpoints{\pgf@tocenter}{\pgf@totarget}%
        \let\pgf@toendangle\pgfmathresult
        \ifdim\pgf@tostartangle pt>\pgf@toendangle pt\relax
          \pgfmathsetmacro\pgf@tostartangle{\pgf@tostartangle-360}%
        \fi
        #2%
          \pgfmathsetmacro\pgf@toendangle{\pgf@toendangle-360}%
        \fi
      \endpgfextra
      arc [radius=+\pgf@toradius, start angle=\pgf@tostartangle, end angle=\pgf@toendangle] \tikztonodes
    }},
  arc through ccw/.style={@arc through={#1}{\iffalse}},
  arc through cw/.style={@arc through={#1}{\iftrue}},
}
\makeatother

\tikzset{middlearrow/.style={
        decoration={markings,
            mark= at position 0.5 with {\arrow{#1}} ,
        },
        postaction={decorate}
    }
}

\makeatletter                                                  %
\def\th@plain{\slshape}                                        %
\makeatother                                                   %
\newcommand{\Cbb}{\mathbb{C}}

\newcommand{\Rbb}{\mathbb{R}}
\newcommand{\Zbb}{\mathbb{Z}}

\newcommand{\p}{_{\ge0}}
\newcommand{\pp}{_{>0}}
\newcommand{\m}{^{-1}}

\newcommand{\argomento}{\,\operatorname{--}\,}
\newcommand{\tr}{\mathrm{tr}}

\usepackage{mathrsfs}

\newcommand{\mc}[1]{\mathcal{#1}}

\newcommand{\abs}[1]{\lvert#1\rvert}
\newcommand{\norm}[1]{\lVert#1\rVert}
\newcommand{\newword}[1]{\textsl{#1}}

\newcommand{\set}[1]{\{ #1 \}}
\newcommand{\ttr}[1]{[#1]}
\newcommand{\cppvector}[2]{\bigl(\begin{smallmatrix}#1\\#2\end{smallmatrix}\bigr)}
\newcommand{\ppmatrix}[4]{\bigl(\begin{smallmatrix}#1&#2\\#3&#4\end{smallmatrix}\bigr)}

\DeclareMathOperator{\PP}{P}

\DeclareMathOperator{\SL}{SL}

\DeclareMathOperator{\im}{im}

\DeclareMathOperator{\arccosh}{arccosh}

\theoremstyle{plain}
\newtheorem{theorem}{Theorem}[section]
\newtheorem{lemma}[theorem]{Lemma}

\theoremstyle{definition}
\newtheorem{definition}[theorem]{Definition}

\newtheorem{example}[theorem]{Example}

\numberwithin{equation}{section}

\begin{document}

\bibliographystyle{plain}

\sloppy

\title[The finiteness conjecture]{The finiteness conjecture holds in
$(\SL_2\Zbb\p)^2$}

\author[G.~Panti, D.~Sclosa]{Giovanni Panti and Davide Sclosa}
\address{Department of Mathematics, Computer Science and Physics\\
University of Udine\\
via delle Scienze 206\\
33100 Udine, Italy}
\email{giovanni.panti@uniud.it}
\email{davide.sclosa@gmail.com}

\begin{abstract}
Let $A,B$ be matrices in $\SL_2\Rbb$ having trace greater than or equal to~$2$. Assume the pair $A,B$ is coherently oriented, that is, can be conjugated to a pair having nonnegative entries. Assume also that either $A,B\m$ is coherently oriented as well, or $A,B$ have integer entries.
Then 
the Lagarias-Wang finiteness conjecture holds for the set $\set{A,B}$, with optimal product in $\set{A,B,AB,A^2B,AB^2}$. In particular, it holds for every matrix pair in $\SL_2\Zbb\p$.
\end{abstract}

\thanks{\emph{2020 Math.~Subj.~Class.}: 05A05; 15A60; 37F32}
\thanks{The first author is partially supported by the 
MIUR Grant E83C18000100006 \emph{Regular and stochastic behaviour in dynamical systems}.}

\maketitle

\section{Introduction}\label{ref1}

Given a finite set $\Sigma$ of square matrices of the same dimension and with real entries, the \newword{joint spectral radius} of $\Sigma$ is
\[
\tilde\rho(\Sigma)=\lim_{n\to\infty}\max\set{\norm{C}:C\in\Sigma^n},
\]
where $\Sigma^n$ is the set of all products of $n$ matrices from $\Sigma$, repetitions allowed, and $\norm{\argomento}$ is the operator norm induced from some vector norm, whose choice is irrelevant.
In short, $\tilde\rho(\Sigma)$ measures the maximal exponential growth rate of vectors under the action of $\Sigma$. Its range of applicability is large and still growing; we refer to~\cite{jungers09}, \cite{guglielmizennaro14} and references therein for a broad panorama and proofs of basic statements.

Despite its simple definition, the computation of the joint spectral radius is a notoriously difficult problem (indeed it is NP-hard~\cite{tsitsiklisblondel97}), even in the restricted form of just determining whether it is nonzero. By the Berger-Wang theorem we have the equivalent characterization
\begin{equation}\label{eq2}
\tilde\rho(\Sigma)=\sup_n\max\set{\rho(C)^{1/n}:C\in\Sigma^n},
\end{equation}
where $\rho(C)$ is the spectral radius of $C$, and in~\cite[p.~19]{lagariaswang95} Lagarias and Wang put forward the \newword{finiteness conjecture}, namely the possibility that the supremum in~\eqref{eq2} is always a maximum. Although in its full generality the conjecture was refuted in~\cite{bouschmairesse02}, counterexamples are difficult to construct, and are widely believed to be rare. The complexity of the matter already appears in the simplest setting, namely sets~$\Sigma$ containing just two $2\times2$ matrices. Indeed, such sets appear in the literature both as finiteness counterexamples~\cite{blondeltheysvladimirov03}, \cite{heremorrissidorovtheys11},
\cite{jenkinsonpollicott18}, \cite{oregon-reyes18},
as well as families of finiteness examples~\cite{jungersblondel08}, \cite{cicone_et_al10}, \cite{kozyakin16}.

In this paper we deal with sets $\Sigma=\set{A,B}$ of matrices in $\SL_2\Rbb$, the group of $2\times2$ matrices with real entries and determinant one. Such matrices act on the hyperbolic plane $\mc{H}=\set{z\in\Cbb:\im z>0}$ via M\"obius isometries $\ppmatrix{a}{b}{c}{d}*z=(az+b)/(cz+d)$, and whenever the group $\Gamma$ generated by $\Sigma$ is fuchsian (i.e., acts on $\mc{H}$ in a properly discontinuous way) the quotient $X=\Gamma\backslash\mc{H}$ is a complete hyperbolic surface. In this case, asking about the joint spectral radius of $\Sigma$ amounts to asking about the maximal \newword{mean free path} along closed geodesics on $X$, namely about the maximal mean time interval between successive crossings of fixed cuts of~$X$ (corresponding to the generators $A,B$ of $\Gamma$) that can be realized among closed geodesics; see~\cite[\S10]{panti20b}. This geometric point of view appears also in~\cite[\S6]{breuillardsert}, where it is discussed the case of two hyperbolic translations well oriented and with disjoint axes (corresponding to $X$ being a pair of pants, provided the two axes are sufficiently far apart). It also appears in~\cite{gekhtmantaylortiozzo19}, although the authors are concerned there with the limiting distribution of mean free paths (which turns out to be gaussian), rather than with their maximal value.

We summarize our results as follows, referring to the following sections for detailed statements.
Fix $\Sigma=\set{A,B}\subset\SL_2\Rbb$ with $\tr(A),\tr(B)\ge2$. We say that $C\in\Sigma^n$ is an \newword{optimal product} if $\tilde\rho(\Sigma)=\rho(C)^{1/n}$ and for no $1\le k<n$ and $D\in\Sigma^k$ we have $\tilde\rho(\Sigma)=\rho(D)^{1/k}$. The existence of optimal products amounts  to the validity of the finiteness conjecture; their  uniqueness ---which may or may not hold--- is intended up to conjugation.
We assume that $A,B$ are coherently oriented; this is a geometric condition (see Definition~\ref{ref2}) that turns out to be equivalent to the fact that $A,B$ are simultaneously conjugate to a pair of nonnegative matrices.
Discarding the trivial case in which $A$ and $B$ commute and hence are simultaneously diagonalizable (or triangularizable, if parabolic),
we will prove the following results.

\begin{itemize}
\item[(I)] If $A$ and $B$ are hyperbolic with intersecting axes, then the one with larger trace is the unique optimal product. If they have the same trace, they are both optimal.
\item[(II)] If $A$ and $B$ are hyperbolic with asymptotically parallel axes, then the one with larger trace is the unique optimal product (this also holds if one of the two is parabolic with fixed point equal to one of the two fixed points of the other). If they have the same trace, then:
\begin{itemize}
\item[(II.1)] If the attracting fixed point of one of the two is repelling for the other, then $A$ and $B$ are both optimal, and no other product is optimal;
\item[(II.2)] Otherwise, 
every product which is not a power is optimal.
\end{itemize}
\item[(III)] If $A$ and $B$ have the same trace and the pair $A,B\m$ is \emph{not} coherently oriented, then $AB$ is the unique optimal product.
\item[(IV)] The above statements leave uncovered the cases in which $A,B$ are both hyperbolic with different trace and ultraparallel axes, or one of the two is parabolic with fixed point distinct from the two fixed points of the other. Assume we are in one of these cases with $\tr(A)<\tr(B)$, and further assume that $A$ and $B$ have integer entries.
\begin{itemize}
\item[(IV.1)] If $\tr(AB)<\tr(B^2)$, then $B$ is the unique optimal product;
\item[(IV.2)] If $\tr(AB)=\tr(B^2)$, then $AB^2$ is the unique optimal product;
\item[(IV.3)] If $\tr(AB) > \tr(B^2)$, then $\tr((AB)^3)$ and $\tr((AB^2)^2)$ differ at least by~$2$; if the former is larger, then $AB$ is the unique optimal product, otherwise so is $AB^2$;
\item[(IV.4)] The statements (IV.1), (IV.2), (IV.3) are false if the assumption about integer entries is dropped.
\end{itemize}
\end{itemize}
Putting together the above statements, we obtain the result stated in the abstract.

This paper is organized as follows: in~\S\ref{ref14} we give the definition of coherent orientation for pairs of nonelliptic matrices in $\SL_2\Rbb$ and prove the equivalence alluded to above. We then establish
in Theorem~\ref{ref3} inequalities relating the translation length of a matrix product with the sum of the translation lengths of the factors. In~\S\ref{ref4} we recast the finiteness property in terms of the existence of maximal elements for a certain preorder defined in the free semigroup on two generators; this interpretation allows us to replace optimal matrix products with better behaved optimal words. We prove statements (I), (II), (III) above in Theorems~\ref{ref7}, \ref{ref9}, and \ref{ref10}. In~\S\ref{ref11} we provide counterexamples and settle~(IV.4). In~\S\ref{ref12} we restrict attention to integer matrices and move from geometric arguments to combinatorial ones, establishing~(IV.1) in Theorem~\ref{ref15}.
The statements (IV.2) and (IV.3) are more involved, requiring a section each, and are established in Theorems~\ref{ref16} and~\ref{ref18}.

\section{Coherently oriented nonelliptic matrices}\label{ref14}

The M\"obius action of $\SL_2\Rbb$ cited in \S\ref{ref1} extends naturally to the euclidean boundary of the hyperbolic plane, namely the real projective line $\PP^1\Rbb=\partial\mc{H}$. A nonidentity matrix $A\in\SL_2\Rbb$ is then \newword{elliptic}, \newword{parabolic}, or \newword{hyperbolic} according to the number of fixed points ---either zero, one, or two--- it has in $\partial\mc{H}$; equivalently, according to the absolute value of its trace being less than, equal to, or greater than~$2$. 
Note that replacing $A$ with $-A$ does not change the action and is irrelevant with respect to anything related to spectral radii.
If $A$ is hyperbolic, one of its fixed points is attracting and we will be denoted by $\alpha^+$, the repelling one being denoted~$\alpha^-$; similar conventions hold for other letters $B,C,\ldots$. If $A$ is parabolic, we agree that $\alpha^+=\alpha^-$ is the only fixed point of $A$.

Let $A$ be a nonidentity matrix in $\SL_2\Rbb$ of trace $\ge2$, and let $d$ denote hyperbolic distance (see~\cite{fenchel89} or~\cite{katok92} for basics of hyperbolic geometry). The \newword{translation length} of $A$ is
\[
\ell(A)=\inf\set{d(z,A*z):z\in\mc{H}}.
\]
It has value $0$ if and only if the infimum is not realized by any $z$, if and only if $A$ is parabolic. If $A$ is hyperbolic, then the set of points $z$ realizing the infimum are precisely those points that lie on the \newword{translation axis} of $A$, namely the unique geodesic of ideal endpoints $\alpha^+$ and
$\alpha^-$. For $A$ as above, spectral radius, trace, and translation length have neat relationships, namely
\begin{equation}\label{eq1}
\begin{aligned}
\rho &=\mathrm{tr}/2+\sqrt{(\mathrm{tr}/2)^2-1} = \exp(\ell/2),\\
\mathrm{tr} &=\rho+\rho\m = 2\cosh(\ell/2),\\
\ell &= 2\arccosh(\mathrm{tr}/2) = 2\log\rho.
\end{aligned}
\end{equation}

Since the functions involved in~\eqref{eq1} are order-preserving bijections between the intervals $[1,\infty)$ (for spectral radius), $[2,\infty)$ (for trace), and $[0,\infty)$ (for translation length), comparing nonelliptic matrices with respect to one of these characteristics is the same as comparing them with respect to any other. Moreover, for every $A,B\in\SL_2\Rbb$ with trace $\ge2$, we have $\rho(A)<\rho(B)$ if and only if $\rho(A^n)<\rho(B^n)$ for some (equivalently, for all) $n\ge1$, and the same statement holds for trace and for translation length.

We look at the ideal boundary $\partial\mc{H}$ as a topological circle, cyclically ordered by the ternary betweenness relation $\alpha\prec \beta\prec \gamma$, which reads ``$\alpha,\beta,\gamma$ are pairwise distinct, and traveling from $\alpha$ to~$\gamma$ counterclockwise we meet $\beta$''. Every pair of distinct points $\alpha,\beta$ determines two closed intervals, namely $[\alpha,\beta]=\set{\alpha,\beta}\cup\set{x:\alpha\prec x\prec \beta}$ and $[\beta,\alpha]=\set{\beta,\alpha}\cup\set{x:\beta\prec x\prec \alpha}$. 

\begin{definition}
Let\label{ref2} $A,B$ be noncommuting matrices in $\SL_2\Rbb$, both with trace greater than or equal to~$2$. If $\alpha^+=\beta^+$, let $I^+=\set{\alpha^+}$. If $\alpha^+\ne\beta^+$, let $I^+$ be the one, of the two intervals $[\alpha^+,\beta^+]$ and $[\beta^+,\alpha^+]$, which is mapped into itself by both $A$ and $B$, if such an interval exists (if it does then it is unique, since $AB\ne BA$ implies $\set{\alpha^+,\alpha^-}\ne\set{\beta^+,\beta^-}$). If such an interval does not exists, leave $I^+$ undefined. Replace in the above lines $A$, $B$ with $A\m$, $B\m$, and $\alpha^+,\beta^+$ with $\alpha^-,\beta^-$, obtaining the definition of~$I^-$. If both of $I^+$ and $I^-$ are defined, then we say that the pair $A,B$ is
\newword{coherently oriented}.
If $A,B$ are coherently oriented, but $A,B\m$ are not, then we say that $A,B$ are \newword{well oriented}.
\end{definition}

It is clear that $A,B$ are coherently oriented if and only if so are $A\m,B\m$. Coherently oriented hyperbolic pairs with ultraparallel axes are necessarily well oriented; see Example~\ref{ref20} and Figure~\ref{fig3} for taxonomy.

\begin{lemma}
Let\label{ref19} $A,B$ be noncommuting matrices in $\SL_2\Rbb$, both with trace $\ge2$. Then they are coherently oriented if and only if there exists $C\in\SL_2\Rbb$ such that $CAC\m$ and $CBC\m$ have nonnegative entries.
\end{lemma}
\begin{proof}
Assume $A,B$ are coherently oriented with $I^+,I^-$ as in Definition~\ref{ref2}. Since $\set{\alpha^+,\alpha^-}\ne\set{\beta^+,\beta^-}$, at least one of $I^+,I^-$ is not a singleton, say $I^+$. Let $K$ be the closure of the complement of $I^+$. Then $\alpha^-,\beta^-\in K$; indeed if, say, $\alpha^-$ were in the interior of $I^+$ we would have $A*\beta^+\notin I^+$, which is impossible.
This fact implies that $A\m[K]\cup B\m[K]\subseteq K$.
Let now $C$ be any matrix in $\SL_2\Rbb$ such that $C[I^+]=[0,\infty]$. Setting $D=CAC\m$ and $E=CBC\m$ we obtain
$D[0,\infty]\cup E[0,\infty]\subseteq [0,\infty]$
and
$D\m[\infty,0]\cup E\m[\infty,0]\subseteq [\infty,0]$.
Write $D=\ppmatrix{a}{b}{c}{d}$; we want to prove that $a,b,c,d\ge0$. Since $D[0,\infty]\subseteq[0,\infty]$, we have that $a$ and $c$ have the same sign, and so do $b$ and $d$. The involution $S=\ppmatrix{}{-1}{1}{}$ exchanges $[0,\infty]$ with $[\infty,0]$. As a consequence,
$SD\m S\m=\ppmatrix{a}{c}{b}{d}$ maps $[0,\infty]$ into itself, which implies that $a$ and $b$ have the same sign, and so do $c$ and $d$.
We conclude that all of $a,b,c,d$ have the same sign, which must be positive, since $\tr(D)\ge2$; the same argument works for $E$.

Conversely, let $A,B$ have nonnegative entries; then $\alpha^+,\beta^+\in[0,\infty]$ and $\alpha^-,\beta^-\in[\infty,0]$. Taking $I^+$ to be the interval of endpoints $\alpha^+,\beta^+$ which is contained in $[0,\infty]$, and analogously for $I^-$, we see that $A,B$ satisfy the conditions of coherent orientation, which are plainly preserved under conjugation.
\end{proof}

\begin{example}
Consider\label{ref20} the following matrices, where zero entries are replaced by spaces:
\begin{gather*}
C=\begin{pmatrix}
9 & 8\\
1 & 1
\end{pmatrix},\quad
D=\begin{pmatrix}
5 & -1\\
1 & 
\end{pmatrix},\quad
E=\frac{1}{10}\begin{pmatrix}
17-2\sqrt{6} & -12+2\sqrt{6} \\
-3-2\sqrt{6} & 8+2\sqrt{6}
\end{pmatrix},\\
G=\begin{pmatrix}
5 & -4\\
4 & -3
\end{pmatrix},\quad
L=\begin{pmatrix}
1 & \\
1 & 1
\end{pmatrix}.
\end{gather*}
We draw in Figure~\ref{fig3} the oriented translation axes of the hyperbolic $C,D,E$, as well as oriented horocycles corresponding to the parabolic $G,L$; note that although we work in the upper-plane model~$\mc{H}$, we draw pictures in the Poincar\'e disk model.

\begin{figure}
\begin{tikzpicture}[scale=2.4]
\coordinate (cp) at (0.2219422, 0.97505980);
\coordinate (cm) at (-0.994356, -0.10609428);

\coordinate (ccontrol) at (0.05555949, 0.56253986);
\coordinate (decontrol) at (0.2087727, -0.01665920);

\coordinate (dp) at (0.3999999, 0.91651513);
\coordinate (dm) at (0.4, -0.9165151);
\coordinate (ep) at (-0.99435607, -0.1060942);
\coordinate (em) at (1,0);

\draw (0,0) circle [radius=1cm];
\draw[middlearrow={latex}] (cm) to[arc through ccw=(ccontrol)] (cp);
\draw[middlearrow={latex}] (dm) to[arc through cw=(decontrol)] (dp);
\draw[middlearrow={latex}] (em) to[arc through ccw=(decontrol)] (ep);
\draw[middlearrow={latex}] (em) arc (0:-360:0.16cm);
\draw[middlearrow={latex}] (0,-1) arc (270:-90:0.16cm);
\node[left] at (cm) {$\gamma^-=\epsilon^+$};
\node at ($(cp) + (0.03,0.14)$)    {$\gamma^+$};

\node at ($(dm) + (0.12,-0.12)$) {$\delta^-$};
\node at ($(dp) + (0.12,0.12)$) {$\delta^+$};

\node[right] at (em) {$\epsilon^-=1$};
\node[below] at (0,-1) {$0$};
\end{tikzpicture}
\caption{Examples of coherently oriented pairs}
\label{fig3}
\end{figure}
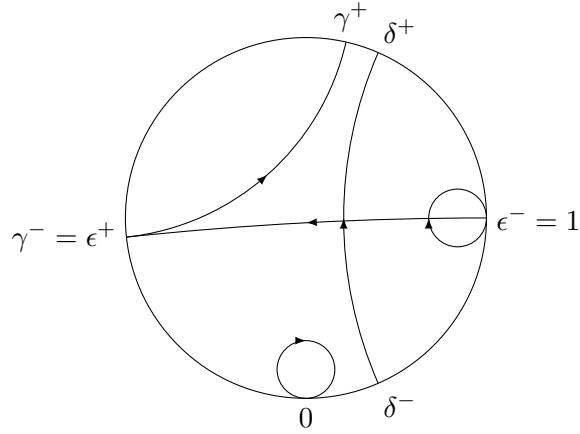

Direct checking shows that coherently oriented pairs can be classified in six subcases as follows, three of them being well oriented.
\begin{itemize}
\item The parabolic-parabolic case, which is necessarily well oriented. The pair $G\m,L$ above is an example (with $I^+=[0,1]$ and $I^-=[1,0]$); note that $G,L$ are not coherently oriented. This case is covered by
Theorem~\ref{ref10}.
\item The parabolic-hyperbolic case, which splits in two. 
A pair may be well oriented (e.g., $C,G$ with $I^+=[1,\gamma^+]$, $I^-=[\gamma^-,1]$), or coherently oriented but not well oriented (e.g., $E,G$).
The first subcase is covered by Theorems~\ref{ref15}, \ref{ref16}, \ref{ref18}, and the second by Theorem~\ref{ref9}.
\item The hyperbolic-hyperbolic case. This splits in three, the subcases of intersecting (such as $D^{\pm1},E^{\pm1}$) or asymptotically parallel axes (such as $C^{\pm1},E^{\pm1}$) being not well oriented. The remaining case, exemplified by $C,D$, is of course well oriented. The first subcase is covered by Theorem~\ref{ref7}, the second by Theorem~\ref{ref9}, and the third by Theorems~\ref{ref15}, \ref{ref16}, \ref{ref18}.
\end{itemize}

We remark that if two matrices in $\SL_2\Zbb$ are coherently oriented, it is not necessarily true that they can be conjugated to a pair with nonnegative \emph{integer} entries; the pair $C,D$ above is one such example.
\end{example}

\begin{lemma}
Let\label{ref8} $A,B$ have trace $\ge2$, and assume they are coherently oriented; let $C$ be a product of $A$ and $B$.
\begin{enumerate}
\item We have $\tr(C)\ge2$, $\gamma^+\in I^+$, $\gamma^-\in I^-$; in particular, if $I^+\cap I^-=\emptyset$ then $C$ is hyperbolic.
\item If $\alpha^+\ne\beta^+$ and $C$ is not a power of $B$, then $\gamma^+\ne\beta^+$; an analogous statement holds for repelling fixed points.
\end{enumerate}
\end{lemma}
\begin{proof}
(1) Surely $\tr(C)>0$ by Lemma~\ref{ref19}. Since $A[I^+]\cup B[I^+]\subseteq I^+$ we have $C[I^+]\subseteq I^+$ and a descending chain $I^+\supseteq C[I^+]\supseteq C^2[I^+]\supseteq\cdots$ that shrinks to $\gamma^+\in I^+$; thus $\tr(C)\ge2$. Inverting both $A$ and $B$ we get $\gamma^-\in I^-$.

(2) We have $C=DAB^k$, for some $k\ge0$ and some product $D$ of $A$ and $B$. Then $\beta^+$ is an endpoint of the interval $B^k[I^+]$, and does not belong to $AB^k[I^+]$. Any further application of $A$ and $B$ 
to $AB^k[I^+]$ leaves $\beta^+$ outside,
and therefore $\beta^+\notin C[I^+]$, which implies $\gamma^+\ne\beta^+$.
\end{proof}

\begin{theorem}
Let\label{ref3} $A,B$ be hyperbolic and coherently oriented  with $I^+\cap I^-=\emptyset$. Then $\ell(AB)$ is less than, equal to, or greater than $\ell(A)+\ell(B)$ if and only if the axes of $A$ and $B$ are intersecting, asymptotically parallel, or ultraparallel, respectively.
\end{theorem}
\begin{proof}
Assume that the axes are asymptotically parallel. Then, possibly replacing $A,B$ with $A\m,B\m$, we may conjugate and assume $\alpha^+=\beta^+=\infty$ (since $I^+$ and $I^-$ do not intersect, $\alpha^+=\beta^-$ is excluded). We then have
\[
A=\begin{pmatrix}
r & t\\
& r\m
\end{pmatrix},\quad
B=\begin{pmatrix}
s & u\\
& s\m
\end{pmatrix},
\]
with $r,s>1$. Since $\ell(A)=2\log r$, and analogously for $B$ and $AB$, we obtain $\ell(AB)=\ell(A)+\ell(B)$.

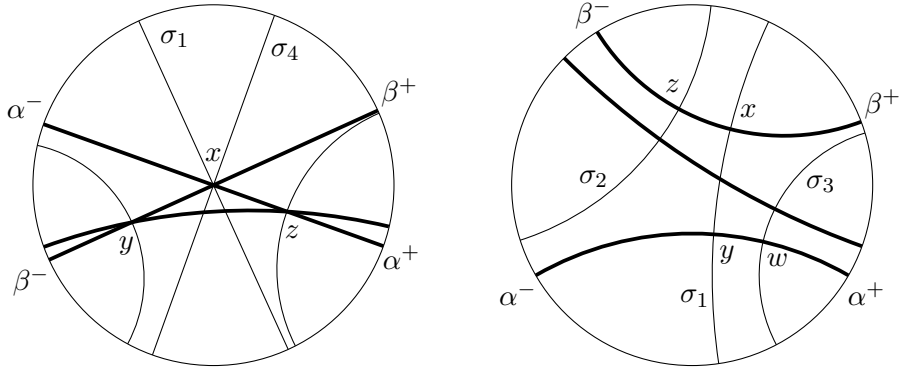
\begin{figure}
\begin{tikzpicture}[scale=2.4]
\coordinate (cp) at (0.94115,-0.33796);
\coordinate (cm) at (-0.94115,0.33796);
\coordinate (dp) at (0.91036,0.41380);
\coordinate (dm) at (-0.91036,-0.41380);
\coordinate (x) at (0,0);
\coordinate (em) at (-0.93956,-0.34236);
\coordinate (ep) at (0.97341,-0.22904);
\coordinate (p1) at (-0.33796,-0.94115);
\coordinate (p3) at (0.33796,0.94115);
\coordinate (p2) at (0.41380, -0.91036);
\coordinate (p4) at (-0.41380, 0.91036);
\coordinate (q2) at (0.91697,0.39892);
\coordinate (q1) at (0.45372,-0.89113);
\coordinate (q3) at (-0.97537,0.22055);
\coordinate (q4) at (-0.47517,-0.87988);
\coordinate (y) at (-0.45205,-0.20547);
\coordinate (z) at (0.40665,-0.14602);

\draw (0,0) circle [radius=1cm];
\draw[line width=0.5mm] (cm) to (cp);
\draw[line width=0.5mm] (dm) to (dp);
\draw[line width=0.5mm] (em) to[arc through cw=(y)] (ep);
\draw (q3) to[arc through cw=(y)] (q4);
\draw (q1) to[arc through cw=(z)] (q2);
\draw (p1) to (p3);
\draw (p2) to (p4);

\node at ($(x) + (0,0.17)$) {$x$};
\node at ($(y) + (-0.03,-0.13)$) {$y$};
\node at ($(z) + (0.03,-0.11)$) {$z$};
\node at ($(cp) + (0.1,-0.06)$) {$\alpha^+$};
\node at ($(dp) + (0.12,0.1)$) {$\beta^+$};
\node at ($(cm) + (-0.1,0.1)$) {$\alpha^-$};
\node at ($(dm) + (-0.1,-0.12)$) {$\beta^-$};
\node at ($(p4) + (0.2,-0.1)$) {$\sigma_1$};
\node at ($(p3) + (0.06,-0.2)$) {$\sigma_4$};
\end{tikzpicture}
\hspace{0.6cm}
\begin{tikzpicture}[scale=2.4]
\coordinate (cp) at (0.86602,-0.5);
\coordinate (cm) at (-0.86602,-0.5);
\coordinate (dp) at (0.93629,0.35121);
\coordinate (dm) at (-0.52662,0.85);
\coordinate (em) at (-0.70912,0.70508);
\coordinate (ep) at (0.94089,-0.33869);

\coordinate (x) at (0.2114481, 0.310447);
\coordinate (q2) at (0.9572039, 0.2894142);
\coordinate (q1) at (0.466348,-0.884601);
\coordinate (q3) at (0.10452, 0.994522);
\coordinate (q4) at (-0.952337,- 0.3050460);
\coordinate (y) at (0.154240, 0.033880);
\coordinate (z) at (0.11888,- 0.272033);
\coordinate (w) at (0.39264, - 0.31304);
\coordinate (t) at (-0.073642, 0.4218477);
\coordinate (u) at (0.14726, - 0.98909);
\coordinate (v) at (0.42332, 0.90598);

\draw (0,0) circle [radius=1cm];
\draw[line width=0.5mm] (cm) to[arc through cw=(z)] (cp);
\draw[line width=0.5mm] (dm) to[arc through ccw=(x)] (dp);
\draw[line width=0.5mm] (em) to[arc through ccw=(y)] (ep);
\draw (q1) to[arc through cw=(w)] (q2);
\draw (q4) to[arc through ccw=(t)] (q3);
\draw (u) to[arc through cw=(y)] (v);
\node at ($(cp) + (0.1,-0.1)$) {$\alpha^+$};
\node at ($(dp) + (0.12,0.1)$) {$\beta^+$};
\node at ($(cm) + (-0.1,-0.1)$) {$\alpha^-$};
\node at ($(dm) + (-0.02,0.1)$) {$\beta^-$};
\node at ($(x) + (0.1,0.1)$) {$x$};
\node at ($(w) + (0.08,-0.1)$) {$w$};
\node at ($(y) + (0.04,-0.42)$) {$y$};
\node at ($(x) + (-0.32,0.24)$) {$z$};
\node at ($(y) + (-0.7,0)$) {$\sigma_2$};
\node at ($(y) + (-0.14,-0.66)$) {$\sigma_1$};
\node at ($(w) + (0.32,0.32)$) {$\sigma_3$};

\end{tikzpicture}

\caption{Coherently oriented geodesics, intersecting case left, nonintersecting right}
\label{fig1}
\end{figure}

Assume the axes intersect; denote by $x$ the intersection point, by $y$ the point at distance $\ell(B)/2$ from $x$ moving towards $\beta^-$, and by $z$ the point at distance $\ell(A)/2$ from $x$ moving towards $\alpha^+$. We sketch the situation in Figure~\ref{fig1} left. Let $\sigma_1$ and $\sigma_2$ be the geodesics perpendicular to the axis $\sigma_3$ of $B$ and passing through $x$ and $y$, respectively. Also, let $\sigma_4$ and $\sigma_5$ be the geodesics perpendicular to the axis $\sigma_6$ of $A$ and passing through $x$ and $z$, respectively. For each $i=1,\ldots,6$, the reflection $S_i$ through $\sigma_i$ is an isometric involution of $\mc{H}$, and we have $A=S_5S_4$ and $B=S_1S_2$. The composition $S_4S_1$ is a rotation about $x$, and equals the composition $S_6S_3$, because the pair $(\sigma_4,\sigma_1)$ is mapped to $(\sigma_6,\sigma_3)$ by a rotation of $\pi/2$ about $x$. Summing up, we obtain
\[
AB=(S_5S_4)(S_1S_2)=S_5(S_4S_1)S_2=S_5(S_6S_3)S_2=(S_5S_6)(S_3S_2),
\]
which is the composition of a rotation of $\pi$ about $y$, followed by a rotation of $\pi$ about $z$. These two rotations leave the geodesic through $y$ and $z$ invariant, and thus this geodesic is the axis of $AB$; moreover,
\[
\ell(AB)=d(y,AB*y)=d(y,S_5S_6*y)=2d(y,z).
\]
By the triangle inequality we conclude
\[
\frac{\ell(AB)}{2}=d(y,z)<d(y,x)+d(x,z)=\frac{\ell(A)}{2}+
\frac{\ell(B)}{2},
\]
as desired.

Assume now that the axes of $A$ and $B$ are ultraparallel; see Figure~\ref{fig1} right. Then they determine a unique common perpendicular, denoted by $\sigma_1$, which intersects the axis of $B$ in $x$ and the axis of $A$ in $y$. Let $z$ be the point at distance $\ell(B)/2$ from $x$ moving towards $\beta^-$, and $w$ the point at distance $\ell(A)/2$ from $y$ towards $\alpha^+$. Draw the perpendicular $\sigma_2$ at $z$ to the axis of $B$, and the perpendicular $\sigma_3$ at $w$ to the axis of $A$. Defining as above $S_i$ to be the reflection of mirror $\sigma_i$, we have $A=S_3S_1$, $B=S_1S_2$, and $AB=S_3S_2$, because $S_1$ cancels. Ultraparallel geodesics have a well-defined hyperbolic distance, still denoted by $d$, and we have $d(\sigma_1,\sigma_2)=\ell(B)/2$, $d(\sigma_1,\sigma_3)=\ell(A)/2$, and $d(\sigma_2,\sigma_3)=\ell(AB)/2$.

Denote by $\overline{\sigma}_i$ the euclidean circle (possibly a straight line) in $\Cbb$ of which $\sigma_i$ is an arc. Then the $\overline{\sigma}_i$s are pairwise nonintersecting (because so are the $\sigma_i$s, and an intersection outside $\mc{H}$ would produce an intersection inside, by M\"obius inversion through $\partial\mc{H}$), and $\overline{\sigma}_1$ separates $\overline{\sigma}_2$ from $\overline{\sigma}_3$, meaning that any circle intersecting $\overline{\sigma}_2$ and $\overline{\sigma}_3$ intersects $\overline{\sigma}_1$ too.
The circles $\partial\mc{H}$ and $\overline{\mathrm{axis}(AB)}$ are distinct and perpendicular to both $\overline{\sigma}_2$ and $\overline{\sigma}_3$; thus the set of circles perpendicular to both 
$\partial\mc{H}$ and $\overline{\mathrm{axis}(AB)}$ constitute the \newword{coaxial pencil} determined by the pair
$\overline{\sigma}_2,\overline{\sigma}_3$~\cite[\S4]{coxeter71}. The key observation here is that $\overline{\sigma}_1$ does not belong to this pencil, since it is not perpendicular to $\mathrm{axis}(AB)$ (because common perpendiculars to pairs of ultraparallel geodesics are unique, and $A$, $B$, $AB$ have distinct axes by Lemma~\ref{ref8}(3)).

Now, while points in the hyperbolic plane obey the triangle inequality, ultraparallel geodesics obey the \newword{non-triangle inequality}~\cite[\S6]{coxeter71}, according to which the distance between ultraparallel geodesics is strictly greater than the sum of distances between the two given geodesics and a third one, separating the two but not coaxial to them. In our case we get
\[
\frac{\ell(AB)}{2}=d(\sigma_2,\sigma_3)>d(\sigma_2,\sigma_1)+d(\sigma_1,\sigma_3)=\frac{\ell(B)}{2}+
\frac{\ell(A)}{2},
\]
again as desired.
\end{proof}

\section{Words}\label{ref4}

As anticipated in~\S\ref{ref1}, some caution is required in defining the length of matrix products. The problem is, of course, that a given pair $A,B\in\SL_2\Rbb$ (even with nonnegative entries) may fail to generate not only a free group ---a tolerable fault--- but even a free semigroup. For example, the matrices
\[
A=\begin{pmatrix}
1 & \sqrt{6} \\
& 1
\end{pmatrix},\quad
B=\begin{pmatrix}
1 & \\
\sqrt{6} & 1
\end{pmatrix},
\]
satisfy the nontrivial identity
\[
A^2B^3A^2=BA^6B=
\begin{pmatrix}
2 & \sqrt{6} \\
\sqrt{6}/2 & 2
\end{pmatrix};
\]
see~\cite{brennercharnow78} for other examples.
We deal with the issue by working with free semigroups of words, as follows.

Let $\set{a,b}$ be a two-letter alphabet, $F_2^+$ the free semigroup of words $w$ of length $\abs{w}\ge1$, and $F_2$ the enveloping free group. Once a pair $A,B\in\SL_2\Rbb$ has been fixed, we consider the group homomorphism $\phi:F_2\to\SL_2\Rbb$ that maps $a$ to $A$ and $b$ to $B$, and the induced character $[\argomento]:F_2\to\Rbb$ defined by $[w]=\tr(\phi(w))$. 

\begin{lemma}
The\label{ref5} following statements are true.
\begin{enumerate}
\item Let $\phi'$, $[\argomento]'$ be induced by another matrix choice $A',B'\in\SL_2\Rbb$. If $[a]=[a]'$, $[b]=[b]'$, $[ab]=[ab]'$, then $[\argomento]=[\argomento]'$.
\item $[w][u]=[wu]+[wu\m]$.
\item Given $w$, let $u$ be either $w\m$, or a rotation of $w$, or the reversal of $w$ (that is, $w$ written backwards). Then $[u]=[w]$.
\item $\ttr{wuv}=\ttr{wv}\ttr{u}-\ttr{wu\m v}$, and thus
$[wu^2v]=[wuv][u]-[wv]$.
\end{enumerate}
\end{lemma}
\begin{proof}
(1) follows from the fact~\cite[Theorem~3.1]{horowitz72} that, for a fixed $w$, there exists a polynomial $f_w\in\Zbb[x,y,z]$ such that, for varying $\phi$, we have $\tr(\phi(w))=f_w\bigl(\tr(\phi(a)),\tr(\phi(b)),\tr(\phi(ab))\bigr)$. 
(2) and the identity $[w]=[w\m]$ are well known, and the invariance of $[\argomento]$ under word rotation follows from the invariance of the trace under conjugation. Define $\phi'(a)=A\m$, $\phi'(b)=B\m$; then $[\argomento]'=[\argomento]$ by (1) and invariance under rotation and group inversion. Letting $u$ be the reversal of $w$, we obtain $[u]=[w\m]'=[w\m]=[w]$, which proves (3). Finally, (4) follows from (2) and rotation invariance.
\end{proof}

A key feature of our formalism is that, not only word length in $F_2^+$ is better behaved than matrix product length, but the implicit comparison of spectral radii in~\eqref{eq2} becomes an explicit preorder on words, as follows.

\begin{definition}
Let $A,B\in\SL_2\Rbb$ have trace greater than or equal to $2$, and assume that they are coherently oriented. Define $\phi$, $\ttr{\argomento}$ as above; by Lemma~\ref{ref8}, $\ttr{\argomento}$ takes values in $\Rbb\p$. We define a binary relation $\preceq_{A,B}$ on $F_2^+$ by
\[
w\preceq_{A,B} u\quad\text{if and only if}\quad\bigl[w^{\abs{u}}\bigr]
\le \bigl[u^{\abs{w}}\bigr].
\]
By saying that a word is \newword{maximal} we mean maximal with respect to $\preceq$ (for simplicity's sake we are dropping in the notation the dependence from $A$ and $B$).
A \newword{complete set of optimal words} is a possibly infinite subset $\set{v_1,v_2,\ldots}$ of $F_2^+$ such that:
\begin{itemize}
\item every $v_i$ is maximal, and is a \newword{Lyndon word}, i.e., is strictly smaller in the lexicographic order than any of its proper rotations (in particular, it is not a power);
\item $v_i\ne v_j$ for $i\ne j$;
\item every maximal $w$ is a power of a rotation of some (necessarily unique) $v_i$.
\end{itemize}
A word that belongs to a complete set of optimal words is an~\newword{optimal} word.
\end{definition}

\begin{lemma}
\begin{enumerate}
\item We\label{ref6} have $w\preceq u$ if and only if $\bigl[w^{m/\abs{w}}\bigr]
\le \bigl[u^{m/\abs{u}}\bigr]$, where $m$ is any common multiple of $\abs{w}$ and $\abs{u}$.
\item The relation $\preceq$ on $F_2^+$ is a preorder, and every two words are comparable.
\item If a complete set of optimal words exists, it is unique.
\item The finiteness conjecture holds for $\Sigma=\set{A,B}$ precisely when $F_2^+$ contains maximal ---equivalently, optimal--- words.
\end{enumerate}
\end{lemma}
\begin{proof}
Let $W=\phi(w)$, and analogously for $u$ and $v$
(later on we will apply this uppercase/lowercase convention without further notice).
By the remarks following Equations~\eqref{eq1}, we have $w\preceq u$ if and only if $\rho\bigl(W^{\abs{u}}\bigr)\le \rho\bigl(U^{\abs{w}}\bigr)$
if and only if 
$\rho\bigl(W^{m/\abs{w}}\bigr)\le \rho\bigl(U^{m/\abs{u}}\bigr)$
if and only if $\bigl[w^{m/\abs{w}}\bigr]\le
\bigl[u^{m/\abs{u}}\bigr]$.
It is clear that $\preceq$ is reflexive and every two words are comparable. Assuming $w\preceq u\preceq v$, and letting $m$ be a common multiple of $\abs{w},\abs{u},\abs{v}$, we obtain
$\bigl[w^{m/\abs{w}}\bigr]\le\bigl[u^{m/\abs{u}}\bigr]\le
\bigl[v^{m/\abs{v}}\bigr]$, and thus $w\preceq v$.
Let $S$ and $S'$ be two complete sets of optimal words, and let $v\in S$. Since $v$ is maximal, it is a power of a rotation of some $v'\in S'$; by the elementary properties of Lyndon words, $v=v'$.
The remaining assertions follow straight from the definitions; note that every word is both greater and less than any of its powers. In particular, if the maximal word $w$ is a power of $u$, then $u$ is maximal as well.
\end{proof}

We can now make precise and prove~(I), (II) and (III) in~\S\ref{ref1}. We stipulate for the rest of this paper, and without further repetitions, that $A,B$ are noncommuting matrices in $\SL_2\Rbb$, of trace greater than or equal to $2$, and coherently oriented.
The following theorem settles~(I).

\begin{theorem}
Let\label{ref7} $A,B$ be both hyperbolic, and assume that the translation axes intersect. If $[a]\le [b]$ then $b$ is an optimal word, and so is $a$ provided $[a]=[b]$. There are no other optimal words.
\end{theorem}
\begin{proof}
We show by induction that for every word $w$ of length $n\ge1$ we have $w\preceq b$. For $n=1$ or $w\in\set{a^n,b^n}$ this is true. Let $n>1$ and $w=au$ without loss of generality, with $u$ not a power of $a$. By Lemma~\ref{ref8}, $\upsilon^+\in I^+\setminus\set{\alpha^+}$ and $\upsilon^-\in I^-\setminus\set{\alpha^-}$. Therefore the axes of $A$ and of $U$ intersect, and by Theorem~\ref{ref3} and inductive hypothesis we have $\ell(W) < \ell(A)+\ell(U)\le n\ell(B)$. Since $\ell(A^n)=n\ell(B)$ if and only if $[a]=[b]$, this also shows uniqueness.
\end{proof}

\begin{theorem}
Let\label{ref9} $B$ be hyperbolic.
\begin{itemize}
\item[(1)] If $A$ is hyperbolic as well and the translation axes are asymptotically parallel, then:
\begin{itemize}
\item[(1.1)] If $[a]\ne[b]$, then the only optimal word is the one among $a$ and $b$ that corresponds to the matrix with larger trace.
\item[(1.2)] If $[a]=[b]$ and $I^+\cap I^-=\emptyset$, then every word which is not a power is optimal. Conversely, if $I^+$ and $I^-$ intersect (necessarily in a singleton), then both $a$ and $b$ are optimal, and there are no other optimal words.
\end{itemize}
\item[(2)] If $A$ is parabolic with $\alpha^+=\alpha^-\in\set{\beta^+,\beta^-}$, then $b$ is the only optimal word.
\end{itemize}
\end{theorem}
\begin{proof}
(1) By conjugating, possibly exchanging $A$ with $B$ and inverting both of them, we may assume $\alpha^+=\beta^+=\infty$
when $I^+\cap I^-=\emptyset$ and $\alpha^+=\beta^-=\infty$ when $I^+\cap I^-\neq \emptyset$.
Let $r=\rho(A)$ and $s=\rho(B)$.
In the first case, after a further conjugation by a parabolic matrix fixing $\infty$, and by a diagonal matrix, we obtain
\[
A=\begin{pmatrix}
r & 1 \\
& r\m
\end{pmatrix},\quad
B=\begin{pmatrix}
s &  \\
& s\m
\end{pmatrix}.
\]
In the second case we similarly obtain
\[
A=\begin{pmatrix}
r & 1 \\
& r\m
\end{pmatrix},\quad
B=\begin{pmatrix}
s\m &  \\
& s
\end{pmatrix}.
\]
It remains to check our claims (1.1) and (1.2) on these two pairs, which  is easily done by direct inspection.
		
(2) is obvious: up to a conjugation we have
\[
A=\begin{pmatrix}
1 & r \\
& 1
\end{pmatrix},\quad
B=\begin{pmatrix}
s &  \\
& s\m
\end{pmatrix},
\]
for some $r\in\Rbb\setminus\set{0}$ and $s\in\Rbb\pp\setminus\set{1}$.
\end{proof}

\begin{theorem}
Let\label{ref10} $\tr(A)=\tr(B)$ and assume that the pair $A,B$ is well oriented. Then the only optimal word is $ab$.
\end{theorem}
\begin{proof}
By possibly exchanging $A$ with $B$, and after an appropriate conjugation, we reduce our matrices to the standard form
\[
A=\begin{pmatrix}
1 &  \\
r & 1
\end{pmatrix},\quad
B=\begin{pmatrix}
1 & r \\
& 1
\end{pmatrix},
\]
for some $r>0$ in the parabolic case, or
\[
A=D\m H D,\quad
B=D H D\m,
\]
in the hyperbolic one; here we set
\[
H=
\begin{pmatrix}
\cosh(\ell/2) & \sinh(\ell/2) \\
\sinh(\ell/2) & \cosh(\ell/2)
\end{pmatrix},\quad
D=\begin{pmatrix}
\exp(d/4) & \\
& \exp(-d/4)
\end{pmatrix},
\]
with $\ell=\ell(A)=\ell(B)>0$ and $d=d(\mathrm{axis}(A),\mathrm{axis}(B))>0$.
We will establish the result by showing that, for every word $u$ which is not a power of $ab$ or or $ba$, we have $u\prec ab$.

Fix such a $u$ of length $n$, and let $n_a,n_b$ be the number of occurrences of $a$ ---respectively $b$--- in it. If one of $n_a,n_b$ is zero, our claim is true: this is clear in the parabolic case (because $AB$ is hyperbolic), and follows from Theorem~\ref{ref3} in the hyperbolic one.

\paragraph{\emph{Claim}} Let $\mc{W}(n_a,n_b)$ be the set of words containing $n_a$ occurrences of $a$ and $n_b$ of $b$, and 
assume without loss of generality $n_a\ge n_b$.
Let $w\in\mc{W}(n_a,n_b)$ with $[w]$ maximal among words in $\mc{W}(n_a,n_b)$. Then every occurrence of $b$ in $w$ is isolated, that is, is preceded and followed, in the cyclic order, by occurrences of $a$.

\paragraph{\emph{Proof of Claim}} In the hyperbolic case this is the content of~\cite[Lemma~5-3]{jorgensensmith90}. The same statement holds in the parabolic case as well. Indeed, the proof of~\cite[Lemma~5-3]{jorgensensmith90} works by repeatedly applying the identity in Lemma~\ref{ref5}(2), while making use of the following facts (references being relative to the quoted paper).
\begin{enumerate}
\item Equation~(5.1), namely
\begin{equation}\label{eq4}
[a^pb^qa^tb^s]-[a^{p+t}b^{q+s}]=pqts\bigl([aba\m b\m]-2\bigr),
\end{equation}
(in the parabolic case, the Chebychev polynomials $\alpha_k,\beta_k$ of~\cite[\S2]{jorgensensmith90} are both equal to $k$).
By explicit computation, in our case we have
\begin{align*}
[a^pb^qa^tb^s] &= 2+(p+t)(q+s)r^2+pqtsr^4,\\
[a^{p+t}b^{q+s}] &= 2+(p+t)(q+s)r^2, \\
[aba\m b\m] &= r^4+2,
\end{align*}
and~\eqref{eq4} remains true.
\item Lemma~5-2, which carries through.
\item Lemma~4-3, which is only used through the inequality $[a^pb^q]>[a^{p-1}b^{q-1}]$; by direct computation one easily checks that this inequality still holds.
\end{enumerate}

Having proved our claim, we may safely assume that all appearances of $b$ in $u$ are isolated. Since by assumption $u$ is not a power of $ab$ or of $ba$, not all occurrences of $a$ are isolated; therefore, up to a rotation, we have
\[
u=a^{k_1}ba^{k_2}\cdots ba^{k_t},
\]
for some $t\ge2$ and $k_1,\ldots,k_t\ge1$.
We shall show that $\upsilon^+$ is in the interior of $I=[0,1]$ and $\upsilon^-$ in the interior of $[-1,0]$.
Let $k\ge1$; it is clear that $A^k[I]\subset I$ both in the parabolic and in the hyperbolic case. We also have $A^kB[I]\subset I$; indeed, it suffice to consider $k=1$. In the parabolic case one easily computes
\[
AB[I]=\biggl[\frac{r}{r^2+1},\frac{r+1}{r^2+r+1}\biggr]\subset I.
\]
The hyperbolic case reduces to the computation of $AB*1$, since $0<AB*0<AB*1$ anyway. Let $AB\cppvector{1}{1}=\cppvector{s}{t}\in\Rbb\pp^2$. Then a little help from SageMath establishes that
\begin{align*}
t-s &= \sinh(-d)+\frac{1}{2}\sinh(d+l)+\frac{1}{2}\sinh(d-l) \\
&= \sinh(-d)+\sinh(d)\cosh(l) \\
&= \sinh(d)\bigl(\cosh(l)-1\bigr)>0,
\end{align*}
and thus $AB*1<1$. As in the proof of Lemma~\ref{ref8} we obtain
$A^{k_1}BA^{k_2}\cdots BA^{k_t}[I]\subset I$, and we conclude that $\upsilon^+$ is in the interior of $I$.

The same argument, applied to the reversal $v$ of $u$, shows that the attracting fixed point of $\phi(v)$ is in the interior of $I$ as well. Letting $J=\ppmatrix{-1}{}{}{1}$, we see that $J$ conjugates $A$ with $A\m$ and $B$ with $B\m$, so that $\phi(v)=J\phi(u\m)J$. This implies that $\upsilon^-$ is the $J$-image of this attracting fixed point, and thus is in the interior of $[-1,0]$.

We now exchange $a$ with $b$ in $u$, obtaining $u'$. This corresponds to conjugating $A$ and $B$ by $F=\ppmatrix{}{1}{1}{}$; in particular, the attracting fixed point of $U'=\phi(u')$ is in the interior of $[1,\infty]$, and the repelling one in $[\infty,-1]$. As a consequence, the translation axes of $U$ and of $U'$ are ultraparallel.

By Lemma~\ref{ref5}(1) we have $\ell(U)=\ell(U')$, and by Theorem~\ref{ref3} $\ell(UU')>2\ell(U)=\ell(U^2)$. We have thus found a word, namely $uu'$, that strictly dominates $u$ in the $\preceq$ preorder. Since $uu'$ contains the same number $n=n_a+n_b$ of occurrences of $a$ and of $b$, we apply again our Claim above and infer $uu'\preceq ab$. This yields $u\prec ab$, as required.
\end{proof}

\section{Counterexamples}\label{ref11}

We have thus proved (I), (II), (III) in~\S\ref{ref1}, covering all cases in which $A,B$ are coherently oriented but not well oriented.
From now on we restrict attention to well oriented
matrices with integer entries, and prove~(IV); our tools, and the overall tone of our paper, will perceptibly move from geometry to combinatorics. This is unavoidable, since finiteness counterexample do indeed exist for well oriented pairs in $\SL_2\Rbb$, wildly popping out as the pair varies smoothly in certain $1$-parameters families of perfectly tame well oriented translations; this is the case, e.g. of the Morris example in~\cite[\S2.2]{oregon-reyes18}. In order to make this tone shift more palatable to the reader, we begin by providing counterexamples, that is, by discussing~(IV.4).

Let us first note that any triple $(r,s,t)\in\Rbb_{\ge2}^3$ determines uniquely up to conjugation a pair $A,B\in\SL_2\Rbb$ 
such that $\tr(A)=r$, $\tr(B)=s$, $\tr(AB)=t$; indeed such triples give coordinates for the Teichm\"uller space of hyperbolic pair of pants.
Let us fix $\tr(A)=101/50$, and vary $\tr(B)=x$ in the interval $[101/50,113/50]$. Adjusting the distance between the axes we may impose that the difference
\[
\Delta=\tr(AB)-\tr(B^2)
\]
be constant, in particular equal to $0$ or to any small positive or negative number; once $\Delta$ is fixed we can compare words in $F_2^+$ in the $\preceq$ order.

We fix $\Delta=0$, so that $\ttr{ab}=\ttr{b^2}=x^2-2$, and compare $ab^2$ with $ab^3$. We must compute $\bigl[(ab^2)^4\bigr]$ and $\bigl[(ab^3)^3\bigr]$.
We have $[ab^2]=[ab][b]-[a]=(x^2-2)x-101/50$ and
$\bigl[(ab^2)^4\bigr]=T_4\bigl([ab^2]\bigr)$, where $T_k(y)$ is the 
degree~$k$ polynomial defined recursively by $T_0(y)=2$, $T_1(y)=y$, $T_k(y)=yT_{k-1}(y)-T_{k-2}(y)$. Thus 
$\bigl[(ab^2)^4\bigr]$ is a polynomial in $x$ of degree $12$, and so is $\bigl[(ab^3)^3\bigr]$. Explicit computation gives
\begin{multline*}
\bigl[(ab^2)^4\bigr]-\bigl[(ab^3)^3\bigr]=
x^{10} - 101/50 x^9 - 9 x^8 + 303/25 x^7\\
+ 98103/2500 x^6 - 909/50 x^5 - 46103/500 x^4 - 2080903/125000 x^3\\
+ 105559/1250 x^2 + 1618727/31250 x + 2050401/6250000,
\end{multline*}
whose graph appears in Figure~\ref{fig2}.
\begin{figure}
\includegraphics[width=7.5cm]{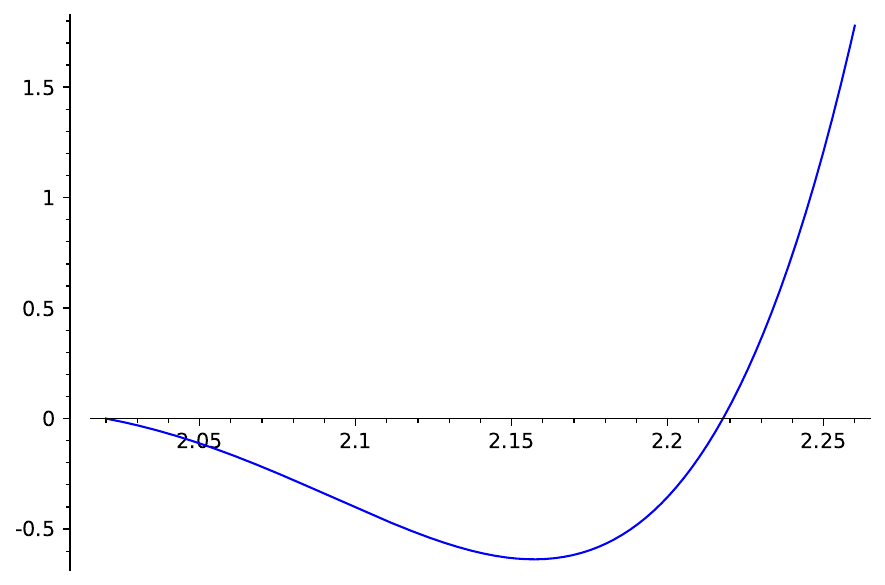}
\caption{Graph of $\bigl[(ab^2)^4\bigr]-\bigl[(ab^3)^3\bigr]$
as a function of $\ttr{b}$.}
\label{fig2}
\end{figure}
Therefore, for $x$ ranging in an appropriate interval, we have $ab^2\prec ab^3$ and the word $ab^2$ is not maximal, contrary to~(IV.2).

Fix now $\Delta=-1/50$, so that $\tr(AB)$ is slightly less than $\tr(B^2)$. Then $[ab^2]=(x^2-2-1/50)x-101/50$, while $[b^3]=T_3(x)=x^3-3x$. Thus, for large $x$, the word $ab^2$ dominates $b$, contrary to~(IV.1).

Finally, let $\Delta=1/50$; by analogous computations we obtain
\begin{align*}
\bigl[(ab)^3\bigr]&=x^6 - 297/50 x^4 + 21903/2500 x^2 - 227799/125000,\\
\bigl[(ab^2)^2\bigr]&=x^6 - 99/25 x^4 - 101/25 x^3 + 9801/2500 x^2 \\
&\quad + 9999/1250 x + 5201/2500,
\end{align*}
which have the same value at $x_0\sim 2.0255364739899213\ldots$.
We compare the two words $ab$ and $ab^2$ with their concatenation $abab^2$ by computing the
differences $\bigl[(ab)^5\bigr]-\bigl[(abab^2)^2\bigr]$ and
$\bigl[(ab^2)^5\bigr]-\bigl[(abab^2)^3\bigr]$, which are polynomials in $x$ of degree $8$ and $13$, respectively. These polynomials are negative at $x_0$, so both $ab$ and $ab^2$ are strictly dominated by $abab^2$, and~(IV.3) fails.

\section{Integer matrices and Case~(IV.1)}\label{ref12}

Since in Theorems~\ref{ref7}, \ref{ref9}, \ref{ref10} we covered the case in which $A$ and $B$ have equal trace, we assume from now on, without loss of generality, that $A,B\in\SL_2\Zbb$ are well oriented and satisfy $2\le\tr(A)<\tr(B)$.
For completeness's sake we provide one specimen for each of the cases (IV.1)--(IV.3) of the Introduction; let~$L$ be as in Example~\ref{ref20} and $N=\ppmatrix{1}{1}{}{1}$. Then we have the following examples:
\begin{align*}
&A=L,B=LN & & &\tr(AB)=4<7=\tr(B^2),\\
&A=LNL,B=NLN^3 & & &\tr(AB)=34=\tr(B^2),\\
&A=L^3N,B=N^2LN^2 & & &\tr(AB)=40>34=\tr(B^2)\\
& & & &\tr((AB)^3)=63880>55223=\tr((AB^2)^2),\\
&A=L^{11},B=LNL & & &\tr(AB)=15>14=\tr(B^2)\\
& & & &\tr((AB)^3)=3330<3362=\tr((AB^2)^2).
\end{align*}

\begin{definition}
A \newword{subword} of the word $w\in F_2^+$ is a possibly empty word obtained from $w$ by deleting one or more not necessarily contiguous letters.
\end{definition}

Since we are now working with matrices having integer entries, the range of $\ttr{\argomento}$ is $\Zbb_{\ge2}$. The remark~(1) in the
following lemma is thus trivial, but key in our proof.

\begin{lemma}
\begin{enumerate}\label{ref13}
\item $\ttr{w}<\ttr{u}$ if and only if $\ttr{w}\le\ttr{u}-1$.
\item If $u$ is a subword of $w$ then $\ttr{u}<\ttr{w}$, exception being made for the case in which $A$ is parabolic and $w$ a power of $a$.
\end{enumerate}
\end{lemma}
\begin{proof}
In order to prove (2) we assume that we are not in the exceptional case, in which $\ttr{u}=\ttr{w}=2$. It suffices to consider the removal of a single letter $c$, which by rotation invariance we may assume being the first one; let then $w=cu$.
If both of $C$ and $U$ are hyperbolic then Theorem~\ref{ref3} applies, yielding
$\ell(W)\ge\ell(C)+\ell(U)$. Thus $\ell(W)>\ell(U)$ and $\ttr{w}>\ttr{u}$ by the remarks following Equation~\eqref{eq1}.

Suppose $C=A$ is parabolic; then, since we are not in the exceptional case, $U$ contains $B$ as a factor and is hyperbolic. Moreover, by Lemma~\ref{ref8}(2), neither of $\upsilon^+,\upsilon^-$ equals the fixed point of $C$.
By Lemma~\ref{ref19} we may assume $C,U\in\SL_2\Rbb\p$, and a further conjugation ---if needed-- by the matrix $F$ in the proof of Theorem~\ref{ref10} reduces us to the case
\[
C=\begin{pmatrix}
1 &  \\
r & 1
\end{pmatrix},\quad
U=\begin{pmatrix}
a & b \\
c & d
\end{pmatrix},
\]
with $r,a,d>0$ and $b,c\ge0$. Now, $c$ can be $0$ ---it is so precisely when one of $\upsilon^+,\upsilon^-$ equals $\infty$--- but $b$ cannot, because otherwise one of $\upsilon^+,\upsilon^-$ would equal the fixed point $0$ of $C$. We thus obtain $\ttr{w}=a+rb+d>a+d=[u]$. An analogous proof applies if $U$ is parabolic.
\end{proof}

\begin{lemma}
	For\label{circle} every integer $s\geq 2$ and every function $f: \mathbb Z/ 2s\mathbb Z \to \mathbb Z$,
	there exists $x$ satisfying both the following inequalities
	\begin{equation*}
		f(x)\geq f(x+2), \quad f(x+1)\le f(x+3).
	\end{equation*}
\end{lemma}
\begin{proof}
	Assume the negation of our statement: it says that $f(x)\geq f(x+2)$ implies $f(x+1)>f(x+3)$ for every $x$.
	This readily leads to a contradiction. Indeed $f(x)<f(x+2)<\cdots$
	cannot always increase; hence there must exist $y$ such that $f(y)\geq f(y+2)$. Repeatedly applying to $f(y)\geq f(y+2)$
	the negation of out statement
	leads to a strictly decreasing sequence $f(y+1) > f(y+3) > f(y+5) >\cdots$, and hence to a contradiction.
\end{proof}

\begin{lemma}
	If\label{1-4} $\ttr{ab}<\ttr{b^2}$ then $\ttr {ab^k}< \ttr{b^{k+1}}$, for every $k\geq 1$.
\end{lemma}
\begin{proof}
We work by induction. The case $k=1$ is by hypothesis, and for $k=2$ we have
\[
\ttr{ab^2} = \ttr{ab}\ttr{b}-\ttr{a}
\le \bigl(\ttr{b^2}-1\bigl)\ttr{b} -\ttr{a}< \ttr{b^3} - 1.
\]
Let $k> 2$; repeatedly applying Lemma~\ref{ref5}(2) to the left side we obtain
\begin{align*}
\ttr{ab^k} &= \ttr{ab}\ttr{b^{k-1}} - \ttr{ab^{2-k}}\\
& = \ttr{ab}\ttr{b^{k-1}} - \ttr{a}\ttr{b^{k-2}} + \ttr{ab^{k-2}}\\
& < \bigl(\ttr{b^2}-1\bigr)\ttr{b^{k-1}} - \ttr{b^{k-2}} + \ttr{b^{k-1}}\\
& = \ttr{b^{k+1}} + \ttr{b^{k-3}} -\ttr{b^{k-1}} - \ttr{b^{k-2}} + \ttr{b^{k-1}}\\
& = \ttr{b^{k+1}} + \ttr{b^{k-3}} - \ttr{b^{k-2}}\\
& \le \ttr{b^{k+1}} -1.
\end{align*}
Here the third line follows by induction hypothesis, and the last one from $\ttr{b^{k-3}}<\ttr{b^{k-2}}$, which is valid for $k>2$.
\end{proof}

\begin{lemma}
Among\label{4-5-6} all words of fixed length, the trace-maximizing ones do not contain the factor $a^2$.
\end{lemma}
\begin{proof}
Since $\ttr{a}<\ttr{b}$, no such word $w$ can be a power of $a$.
Assume that $w$ contains $a^2$. Then, up to a rotation, $w=bua^2$, and it is enough to prove $\ttr {bua^2}< \ttr{buab}$.
The right side equals $\ttr{bua} \ttr{b}-\ttr{ua}$, and the other side $\ttr{bua} \ttr{a}-\ttr{bu}$. The difference is
then greater than
$\ttr{bua}-\ttr{ua}+\ttr{bu}$, which is strictly positive by Lemma~\ref{ref13}(2).
\end{proof}

\begin{lemma} \label{41}
If $\ttr{ab}<\ttr{b^2}$ then, for every $s,k_1,\cdots,k_s\geq 1$,
we have
\begin{equation*}
\ttr {ab^{k_1}\cdots ab^{k_s}} < \ttr {b^{k_1+\cdots +k_s+s}}.
\end{equation*}
\end{lemma}
\begin{proof}
	We work by induction on $s$, the case $s=1$ having been proved in Lemma \ref{1-4}. Let $s\geq 2$. We can
	suppose $k_{s-1}\geq k_{s}$, which ensures that
	\begin{align*}
		\ttr {ab^{k_1}\cdots ab^{k_{s-1}}(ab^{k_s})^{-1}}=\ttr{ b^{k_1}\cdots ab^{k_{s-1}-k_s}}
	\end{align*}
	is positive. We thus obtain
	\begin{align*}
		\ttr {ab^{k_1}\cdots ab^{k_s}} &<\ttr {ab^{k_1}\cdots ab^{k_{s-1}}} \ttr {ab^{k_s}}\\
		& \le (\ttr{b^{k_1+\cdots +k_{s-1} + s-1}}-1)(\ttr{b^{k_s+1}}-1)\\
		& = \ttr{b^{k_1+\cdots +k_{s-1} + k_s + s}} + \ttr{b^{k_1+\cdots +k_{s-1} - k_s + s-2}}\\
		& \qquad -  \ttr{b^{k_1+\cdots +k_{s-1} + s-1}} -(\ttr{b^{k_s+1}}-1)\\
		&< \ttr{b^{k_1+\cdots +k_{s-1} + k_s + s}}-(\ttr{b^{k_s+1}}-1)\\
		&< \ttr{b^{k_1+\cdots +k_{s-1} + k_s + s}}.
	\end{align*}
\end{proof}

We can now prove~\S\ref{ref1}(IV.1).

\begin{theorem}
Let\label{ref15} $A,B\in\SL_2\Zbb\p$, and assume $\tr(A)<\tr(B)$ and $\tr(AB)<\tr(B^2)$. The $b$ is the only optimal word.
\end{theorem}
\begin{proof}
Let $w$ be a word of length $n$, trace-maximizing among all words of the same length; we must prove that $w$ does not contain any $a$.
By Lemma~\ref{4-5-6}, $w$ does not contain $a^2$ as a factor. After a rotation we may apply Lemma~\ref{41}, and conclude that $w=b^n$.
\end{proof}

\section{Case (IV.2)}

The remaining cases~(IV.2) and (IV.3) are more involved and require a section each. The standing assumptions in this section for $A,B\in\SL_2\Zbb\p$ are $\tr(A)<\tr(B)$ and $\tr(AB)=\tr(B^2)$. They yield
\begin{equation*}
\ttr{ab} - \ttr{a}\ttr{b} = \ttr{b^2} -\ttr{a}\ttr{b} \geq  \ttr{b^2} - (\ttr{b}-1)\ttr{b}
= \ttr{b} - 2  \geq \ttr{a}-1 > 0,
\end{equation*}
or, equivalently, $\ttr{ab^{-1}}<0$. This will be useful several times.

\begin{lemma}
Fix\label{5A} a word $w$ and assume $s\geq 0$. Then we have:
	\begin{enumerate}
		\item $\ttr{wab(ab^2)^sab^3} < \ttr{wab^2(ab^2)^s ab^2}$;
		\item $\ttr{wab^3(ab^2)^sab} < \ttr{wab^2(ab^2)^s ab^2}$, if $w$ is empty or begins with $a$.
	\end{enumerate}
\end{lemma}
\begin{proof}
We prove (1). We have
\begin{align*}
\ttr{wab(ab^2)^sab^3}&= \ttr{wab(ab^2)^sab}\ttr{b^2} -
\ttr{wab(ab^2)^sab^{-1}}\\
&=\ttr{wab(ab^2)^sab}\ttr{b}^2-2\ttr{wab(ab^2)^sab}\\
&\quad -\ttr{wab(ab^2)^sa}\ttr{b} + \ttr{wab(ab^2)^sab}\\
&=\ttr{wab(ab^2)^sab}\ttr{b}^2 - \ttr{wab(ab^2)^sab} - \ttr{wab(ab^2)^sa}\ttr{b},
\end{align*}
and, by Lemma~\ref{ref5}(4),
\begin{align*}
\ttr{wab^2(ab^2)^sab^2} &= \ttr{wab(ab^2)^sab^2}\ttr{b} - \ttr{wa(ab^2)^sab^2}\\
& =  \ttr{wab(ab^2)^sab}\ttr{b}^2 - \ttr{wab(ab^2)^sa}\ttr{b} - \ttr{wa(ab^2)^s ab^2}.
\end{align*}
Subtracting the first end result from the second we get
\begin{align*}
\ttr{wab(ab^2)^sab} - \ttr{wa(ab^2)^sab^2} 
    & = \ttr{w(ab^2)^sab}\ttr{ab} -\ttr{w(ab)\m(ab^2)^s ab}\\
    &\quad - \ttr{wa(ab^2)^sab}\ttr{b} +\ttr{wa(ab^2)^sa} \\
    & =	\ttr{w(ab^2)^sab}\bigl(\ttr{ab}-\ttr{a}\ttr{b}\bigr) -\ttr{wb(ab^2)^{s-1}ab} \\
    &\qquad  + \ttr{wb^2(ab^2)^{s-1}ab}\ttr{b} +\ttr{wa(ab^2)^sa}.			
	\end{align*}
If $s\geq 1$ this is strictly positive by the observation preceding the lemma and Lemma~\ref{ref13}(2).
This also holds when $s=0$, since the sum of the two middle terms
becomes $-\ttr{w}+\ttr{wb}\ttr{b}>0$.

The proof of (2) is similar, except that in the second expansion we work on the last $ab^2$. We have
\begin{align*}
\ttr{wab^3(ab^2)^sab}
& = \ttr{wab(ab^2)^sab}\ttr{b^2} - \ttr{wab^{-1}(ab^2)^sab}\\
& = \ttr{wab(ab^2)^sab}\ttr{b}^2 - 2\ttr{wab(ab^2)^sab}\\
& \quad -\ttr{wa(ab^2)^sab}\ttr{b} + \ttr{wab(ab^2)^sab}\\
& = \ttr{wab(ab^2)^sab}\ttr{b}^2 - \ttr{wab(ab^2)^sab} - \ttr{wa(ab^2)^sab}\ttr{b},
\end{align*}
and
\begin{align*}
\ttr{wab^2(ab^2)^sab^2} &= \ttr{wab^2(ab^2)^sab}\ttr{b} - \ttr{wab^2(ab^2)^sa}\\
& =  \ttr{wab(ab^2)^sab}\ttr{b}^2 - \ttr{wa(ab^2)^sab}\ttr{b} - \ttr{wab^2(ab^2)^s a}.
\end{align*}
As above, subtracting the two end results we get
\begin{align*}
\ttr{wab&(ab^2)^sab} - \ttr{wab^2(ab^2)^s a} \\
& = \ttr{wab(ab^2)^sab} - \ttr{wab(ab^2)^s a}\ttr{b} + \ttr{wa(ab^2)^s a}\\
& = \ttr{wab(ab^2)^sa}\ttr{b} - \ttr{wab(ab^2)^sab^{-1}} - \ttr{wab(ab^2)^s a}\ttr{b} + \ttr{wa(ab^2)^s a}\\
& = \ttr{wa(ab^2)^s a} - \ttr{wab(ab^2)^sab^{-1}} \\
& = \ttr{wa(ab^2)^s a} - \ttr{wab(ab^2)^s}\ttr{ab^{-1}} + \ttr{wab(ab^2)^sb a^{-1}}.
\end{align*}
Since $\ttr{ab^{-1}}<0$ and $w$ is empty or beginning with $a$, this is positive.
\end{proof}

\begin{lemma}
Under\label{5B} the same hypotheses of Lemma~\ref{5A} we have:
\begin{enumerate}
\item $\ttr{wab(ab^2)^sab^4} < \ttr{wab^2(ab^2)^s ab^3}$,
\item $\ttr{wab^4(ab^2)^sab} < \ttr{wab^3(ab^2)^s ab^2}$, if $w$ is empty or begins with $a$.
\end{enumerate}
\end{lemma}
\begin{proof}
(1) follows by Lemma~\ref{5A}(1), applied to the word $bw$. 
(2) Write $\widetilde w$ for the reversal of $w$. Then by Lemma~\ref{ref5}(3) we obtain
\begin{align*}
\ttr{wab^4(ab^2)^sab} &= \ttr{ba(b^2a)^sb^4a\widetilde{w}}\\
&= \ttr{a\widetilde{w}b(ab^2)^sab^4}\\
&< \ttr{a\widetilde{w}b^2(ab^2)^sab^3}
\tag{\text{by (1)}}\\
&= \ttr{wab^3a(b^2a)^sb^2}\\
&= \ttr{wab^3(ab^2)^sab^2}.
\end{align*}
Note that the use of~(1) in the third line is legitimate, since
$\widetilde{w}$ ends with $a$, or is empty.
\end{proof}

\begin{lemma}
Let\label{5C} $w$ be a word that is empty or ends with $b$, and let $k,h\geq 0$.
Then we have
\[
\ttr{ab^2wab(ab^2)^kab(ab^2)^hab} < \ttr{ab^2w(ab^2)^{k+h+2}}.
\]
\end{lemma}
\begin{proof}
On the left hand side we have
\begin{equation}\label{eq5}
\begin{split}
\ttr{ab^2w&ab(ab^2)^kab(ab^2)^hab} \\
&= \ttr{bwab(ab^2)^kab(ab^2)^hab}\ttr{ab} - \ttr{bwab(ab^2)^kab(ab^2)^h}\\
&= \ttr{bwab(ab^2)^k(ab^2)^hab}\ttr{ab}^2 - \ttr{bwab(ab^2)^kb(ab^2)^{h-1}ab}\ttr{ab} \\
& \qquad - \ttr{bwab(ab^2)^kab(ab^2)^h}\\
&= \ttr{bw(ab^2)^k(ab^2)^hab}\ttr{ab}^3 - \ttr{bwb(ab^2)^{k+h-1}ab}\ttr{ab}^2\\
& \qquad - \ttr{bwab(ab^2)^kb(ab^2)^{h-1}ab}\ttr{ab} - \ttr{bwab(ab^2)^kab(ab^2)^h}\\
&= \ttr{w(ab^2)^{k+h+1}}\ttr{ab}^3 - \ttr{wb(ab^2)^{k+h}}\ttr{ab}^2\\
& \qquad - \ttr{wab(ab^2)^kb(ab^2)^{h}}\ttr{ab} - \ttr{bwab(ab^2)^kab(ab^2)^h}.
\end{split}
\end{equation}
On the other side we have
\begin{align*}
\ttr{ab^2w(ab^2)^{k+h+2}}
&= \ttr{ab^2w(ab^2)^{k+h}}\ttr{(ab^2)^2}- \ttr{ab^2w(ab^2)^{k+h-2}}\\
&= \ttr{w(ab^2)^{k+h+1}}\ttr{(ab^2)^2}- \ttr{w(ab^2)^{k+h-1}}.
\end{align*}
It is enough to show that $\ttr{w(ab^2)^{k+h-1}}< \ttr{wb(ab^2)^{k+h}}\ttr{ab}^2$ and that $\ttr{ab}^3<\ttr{(ab^2)^2}$.
If $k+h\geq 1$ the first inequality is clear. If $k=h=0$, it amounts to
\[
\ttr{wb^{-1}(ab)^{-1}}<\ttr{wb}\ttr{ab}^2,
\]
or equivalently
\[
\ttr{wb^{-1}}\ttr{ab}<\ttr{wb}\ttr{ab}^2+\ttr{wb^{-1}ab},
\]
which holds, since $w$ is empty or ends with $b$.
	
We now show $\ttr{ab}^3<\ttr{(ab^2)^2}$. We have
\begin{equation*}
\ttr{ab}^3 = \ttr{b^2}^3 
 = \ttr{b^2}\bigl(\ttr{b^4}+2\bigr)
 = \ttr{b^6}+3\ttr{b^2}.
\end{equation*}
Let $\ttr{b} - \ttr{a}=\Delta \ge1$; then
\begin{multline*}
\ttr{(ab^2)^2} = \ttr{ab^2}^2-2
 = \bigl(\ttr{ab}\ttr{b}-\ttr{a}\bigr)^2-2
 = \bigl(\ttr{b^2}\ttr{b}-\ttr{a}\bigr)^2-2\\
 = \bigl(\ttr{b^3}+\Delta\bigr)^2-2
 = \ttr{b^3}^2-2 +2\ttr{b^3}\Delta + \Delta^2
 = \ttr{b^6}+2\ttr{b^3}\Delta + \Delta^2.
\end{multline*}
We thus have to show $2\ttr{b^3}\Delta + \Delta^2> 3\ttr{b^2}$,
and it is enough
to prove the case $\Delta =1$,
namely $2\ttr{b^3}-3\ttr{b^2}+1>0$.
We compute
\begin{align*}
2\ttr{b^3}-3\ttr{b^2}+1 &= 2T_3\bigl(\ttr{b}\bigr)-
    3T_2\bigl(\ttr{b}\bigr)+1 \\
& = 2\ttr{b}^3 -3\ttr{b}^2 -6\ttr{b} + 7.
\end{align*}
Since the polynomial $2x^3-3x^2-6x+7$ has three real roots, all of them strictly less than $3$, and the trace of $B$ is at least $3$, the desired inequality follows.
\end{proof}

\begin{lemma}
Let\label{5D} $w$ be a word that is empty or begins with $a$, and let $s\geq 0$.
Then we have:
\begin{enumerate}
\item $\ttr{wab^4(ab^2)^sab^3}<\ttr{w(ab^2)^{s+3}}$;
\item $\ttr{wab^3(ab^2)^sab^4}<\ttr{w(ab^2)^{s+3}}$.
\end{enumerate}
\end{lemma}
\begin{proof}
	We prove (1). On the left side we have
	\begin{align*}
		\ttr{wab^4&(ab^2)^sab^3}\\
		& = \ttr{wab^2(ab^2)^sab^3}\ttr{b^2} - \ttr{wa(ab^2)^sab^3}\\
		& = \ttr{wab^2(ab^2)^sab^2}\ttr{b}\ttr{b^2} - \ttr{wab^2(ab^2)^sab}\ttr{b^2} - \ttr{wa(ab^2)^sab^3}\\
		& = \ttr{w(ab^2)^{s+2}}\ttr{b}\ttr{ab} - \ttr{w(ab^2)^{s+1}ab}\ttr{b^2}- \ttr{wa(ab^2)^{s+1}b}\\
		& = \ttr{w(ab^2)^{s+2}}\ttr{ab^2} + \ttr{w(ab^2)^{s+2}}\ttr{a} \\
		&\qquad - \ttr{w(ab^2)^{s+1}ab}\ttr{b^2}- \ttr{wa(ab^2)^{s+1}b},
	\end{align*}	
while on the other side we have
\[
\ttr{w(ab^2)^{s+3}} = \ttr{w(ab^2)^{s+2}}\ttr{ab^2} - \ttr{w(ab^2)^{s+1}}.
\]
We obtain $\ttr{w(ab^2)^{s+1}} < \ttr{wa(ab^2)^{s+1}b}$
from Lemma~\ref{ref13}(2). We complete the proof by computing
\begin{align*}
		\ttr{w&(ab^2)^{s+1}ab}\ttr{b^2} - \ttr{w(ab^2)^{s+2}}\ttr{a} \\
		& = \ttr{w(ab^2)^{s+1}ab}\ttr{ab} - \ttr{w(ab^2)^{s+1}abab}-  \ttr{w(ab^2)^{s+1}aba^{-1}b}\\
		& = \ttr{w(ab^2)^{s+1}abab} + \ttr{w(ab^2)^{s+1}}- \ttr{w(ab^2)^{s+1}abab}-  \ttr{w(ab^2)^{s+1}aba^{-1}b}\\
		& = \ttr{w(ab^2)^{s+1}} - \ttr{w(ab^2)^{s+1}aba^{-1}b}\\
		& = \ttr{w(ab^2)^{s+1}} - \ttr{w(ab^2)^{s+1}ab}\ttr{a^{-1}b} + \ttr{w(ab^2)^{s+1}a^2},
\end{align*}
whose end result is positive since $\ttr{a^{-1}b}<0$.

(2) can be obtained from (1) as in the proof of Lemma~\ref{5B}.
\end{proof}

\begin{lemma}
Let\label{5E} $w$ be a word that is empty or ends with $b$, and let $k,h\geq 0$.
Then we have
\[
\ttr{wab^4(ab^2)^kab^4(ab^2)^hab^4} < \ttr{w(ab^2)^{k+h+5}}.
\]
\end{lemma}
\begin{proof}
On the left side we have
\begin{align*}
		&\ttr{wab^4(ab^2)^kab^4(ab^2)^hab^4}\\
		& = \ttr{wab^4(ab^2)^kab^4(ab^2)^hab^2}\ttr{b^2} - \ttr{wab^4(ab^2)^kab^4(ab^2)^ha}\\
		& = \ttr{wab^4(ab^2)^kab^2(ab^2)^hab^2}\ttr{b^2}^2
         - \ttr{wab^4(ab^2)^ka(ab^2)^hab^2}\ttr{b^2}\\
		& \qquad -\ttr{wab^4(ab^2)^kab^4(ab^2)^ha}\\
		& = \ttr{wab^2(ab^2)^kab^2(ab^2)^hab^2}\ttr{b^2}^3
         - \ttr{wa(ab^2)^kab^2(ab^2)^hab^2}\ttr{b^2}^2\\
		& \qquad - \ttr{wab^4(ab^2)^ka(ab^2)^hab^2}\ttr{b^2} -\ttr{wab^4(ab^2)^kab^4(ab^2)^ha}\\
		& = \ttr{w(ab^2)^{k+h+3}}\ttr{b^2}^3
         - \ttr{wa(ab^2)^{k+h+2}}\ttr{b^2}^2 \\
		& \qquad - \ttr{wab^4(ab^2)^ka(ab^2)^{h+1}}\ttr{b^2} -\ttr{wab^4(ab^2)^kab^4(ab^2)^ha}.
	\end{align*}
On the other side we have
\[
\ttr{w(ab^2)^{k+h+5}} = \ttr{w(ab^2)^{k+h+3}}\ttr{(ab^2)^2}
 - \ttr{w(ab^2)^{k+h+1}}.
\]
The second end result is strictly greater than the first, because
$\ttr{w(ab^2)^{k+h+1}}<\ttr{wa(ab^2)^{k+h+2}}\ttr{b^2}^2$ by Lemma~\ref{ref13}(2), and $\ttr{b^2}^3=\ttr{ab}^3<\ttr{(ab^2)^2}$ by the proof of Lemma~\ref{5C}.
\end{proof}

\begin{lemma}
Let\label{5F} $w$ be a word that begins with $ab$ or is empty, and let $k,h\geq 0$. Then
\[
\ttr{wab^3(ab^2)^kab^3(ab^2)^hab^3}<\ttr{w(ab^2)^{k+h+4}}.
\]
\end{lemma}
\begin{proof}
This time we work on the term on the right hand side; the step in the middle of the following identity chain results from $\ttr{(ab)b^2}+\ttr{(ab)b^{-2}}=\ttr{ab}\ttr{b^2}=\ttr{ab}^2$.
\begin{align*}
		\ttr{w(a&b^2)^{k+h+4}} \\
		& = \ttr{wab(ba)bb(ab^2)^{k+h+2}} \\
		& = \ttr{wab^3(ab^2)^{k+h+2}}\ttr{ba} - \ttr{waba^{-1}b(ab^2)^{k+h+2}}\\
		& = \ttr{wab^3(ab^2)^k ab^2 (ab)b (ab^2)^h}\ttr{ab}- \ttr{waba^{-1}b(ab^2)^{k+h+2}}\\
		& = \ttr{wab^3(ab^2)^k ab^3 (ab^2)^h}\ttr{ab}^2 - \ttr{wab^3(ab^2)^k aba^{-1}b (ab^2)^h}\ttr{ab}\\
		& \qquad - \ttr{waba^{-1}b(ab^2)^{k+h+2}}\\
		& = \ttr{wab^3(ab^2)^k ab^3 (ab^2)^h}\bigl(\ttr{ab^3}+\ttr{ab^{-1}}\bigr) \\
		& \qquad - \ttr{wab^3(ab^2)^k aba^{-1}b (ab^2)^h}\ttr{ab} - \ttr{waba^{-1}b(ab^2)^{k+h+2}}\\
		& = \ttr{wab^3(ab^2)^k ab^3 (ab^2)^hab^3} + \ttr{wab^3(ab^2)^k ab^3 (ab^2)^{h-1}ab^{-1}a^{-1}}\\
		& \qquad +\ttr{wab^3(ab^2)^k ab^3 (ab^2)^h}\ttr{ab^{-1}} \\
		& \qquad - \ttr{wab^3(ab^2)^k aba^{-1}b (ab^2)^h}\ttr{ab} - \ttr{waba^{-1}b(ab^2)^{k+h+2}}.
	\end{align*}
Now, the first summand of the end result is the left side of the desired inequality, and the second is positive due to our hypotheses on $w$. We develop the fourth summand, in order to make the third appear:
\begin{align*}
\ttr{wab^3&(ab^2)^k aba^{-1}b (ab^2)^h}\ttr{ab}\\
& = \ttr{wab^3(ab^2)^k ab (ab^2)^h}\ttr{ab}\ttr{a^{-1}b}
-\ttr{wab^3(ab^2)^k a^2 (ab^2)^h}\ttr{ab}\\
& = \ttr{wab^3(ab^2)^k ab (ab^2)^h}\ttr{b^2}\ttr{a^{-1}b}
-\ttr{wab^3(ab^2)^k a^2 (ab^2)^h}\ttr{ab}\\
& = \ttr{wab^3(ab^2)^k ab^3 (ab^2)^h}\ttr{a^{-1}b}
+\ttr{wab^3(ab^2)^k ab^{-1} (ab^2)^h}\ttr{a^{-1}b}\\
& \qquad -\ttr{wab^3(ab^2)^k a^2 (ab^2)^h}\ttr{ab} \\
& = \ttr{wab^3(ab^2)^k ab^3 (ab^2)^h}\ttr{ab^{-1}} \\
& \qquad +\ttr{wab^3(ab^2)^k a (ab^2)^h}\ttr{a^{-1}b}\ttr{b} - \ttr{wab^3(ab^2)^k ab (ab^2)^h}\ttr{a^{-1}b}\\
& \qquad - \ttr{wab^3(ab^2)^k a^2 (ab^2)^h}\ttr{ab}.
\end{align*}

We are then left with proving that the sum
\begin{multline*}
-\ttr{wab^3(ab^2)^k a (ab^2)^h}\ttr{a^{-1}b}\ttr{b} + \ttr{wab^3(ab^2)^k ab (ab^2)^h}\ttr{a^{-1}b}\\
+\ttr{wab^3(ab^2)^k a^2 (ab^2)^h}\ttr{ab}
-\ttr{wab(ab^2)^{k+h+2}}\ttr{a^{-1}b}+ \ttr{wa^2(ab^2)^{k+h+2}}
\end{multline*}
is positive. The first, third, and last summand surely are, and so is the sum or the second with the fourth, because
\[
\ttr{wab^3(ab^2)^k ab (ab^2)^h} < \ttr{w(ab^2)^{k+h+2}}
<\ttr{wab(ab^2)^{k+h+2}},
\]
by Lemma~\ref{5A}(2) and Lemma~\ref{ref13}(2).
\end{proof}

We finally arrive at~\S\ref{ref1}(IV.2).

\begin{theorem}
Let\label{ref16} $A,B\in\SL_2\Zbb\p$, and assume $\tr(A)<\tr(B)$ and $\tr(AB)=\tr(B^2)$. The $ab^2$ is the only optimal word.
\end{theorem}
\begin{proof}
Let $u$ be a word which is trace-maximizing among words of the same length; by the remarks following Equation~\eqref{eq1} we may assume, possibly replacing $u$ with its cube, that this length is a multiple of $3$.
We have to prove that $u$ is a power of a rotation of $ab^2$.	
By Lemma \ref{4-5-6}, $u$ does not contain $a^2$ as a factor, up to rotations. We have
	\begin{align*}
		\ttr{wab^5}= & 	\ttr{wab^2}\ttr{b^3}-\ttr{wab^{-1}}\\
		&= \ttr{wab^2}\bigl(\ttr{bb}\ttr{b}-\ttr{b}\bigr)-\ttr{wab^{-1}}\\
		&= \ttr{wab^2}\bigl(\ttr{ab}\ttr{b}-\ttr{b}\bigr)-\ttr{wab^{-1}}\\
		&\le \ttr{wab^2}\bigl(\ttr{ab}\ttr{b}-\ttr{a}-1\bigr)-\ttr{wab^{-1}}\\
		&= \ttr{wab^2}\bigl(\ttr{ab^2}-1\bigr)-\ttr{wab^{-1}}\\
		&= \ttr{w(ab^2)^2}+\ttr{w} -\ttr{wab^2}-\ttr{wab^{-1}}\\
		&= \ttr{w(ab^2)^2}+\ttr{w}-\ttr{wab^2} -\ttr{wa}\ttr{b}+\ttr{wab}\\
		&< \ttr{w(ab^2)^2},
	\end{align*}
and therefore $b^5$ is also excluded. Moreover
\[
\ttr{b^3} = \ttr{bb}\ttr{b}-\ttr{b}
 = \ttr{ab}\ttr{b}-\ttr{b}
 = \ttr{ab^2} +\ttr{a}-\ttr{b}
 < \ttr{ab^2}
\]
shows that $u$ is not a power of $b$. Summing up, $u$ uniquely factorizes as a product of $ab$, $ab^2$, $ab^3$, and $ab^4$;
we refer to this as the \newword{syllabic decomposition} of $u$.

Suppose that $ab$ occurs as a syllable (occurrences are always intended up to rotations). Since the length is a multiple of $3$, at least one of the following cases must hold:
\begin{itemize}
\item one of $ab^3$ and $ab^4$ occurs as well,
\item $ab$ occurs at least thrice.
\end{itemize}
These occurrences will be separated by zero or more occurrences of $ab^2$. In any case, Lemmas~\ref{5A}, \ref{5B}, and~\ref{5C} apply, and $u$ is not trace-maximizing. Thus the syllable $ab$ does not occur in $u$.

Suppose $ab^4$ occurs. Then
\begin{itemize}
\item either $ab^3$ occurs as well,
\item or $ab^4$ occurs at least thrice.
\end{itemize}
Lemmas \ref{5D} and \ref{5E} treat these cases, and thus exclude $ab^4$.

We have established that the only syllables occurring in $u$ are $ab^2$ and $ab^3$. If the latter occurs,  
it must do so at least thrice, and Lemma \ref{5F} applies.
We are then left with occurrences of $ab^2$ only, and the proof is complete.
\end{proof}

\section{Case (IV.3)}\label{ref17}

The standing assumptions in this final section are that $A,B\in\SL_2\Zbb$ satisfy $2\le\tr(A)<\tr(B)$ and $\tr(AB)>\tr(B^2)$.
First, a dichotomy.

\begin{lemma}\label{dichotomy}
Under the above assumptions the numbers $\ttr{(ab)^3}$ and $\ttr{(ab^2)^2}$ differ by at least $2$.
\end{lemma}
\begin{proof}
We work modulo $\ttr{ab}$. We have
\[
\ttr{(ab)^3} \equiv  \ttr{(ab)^2}\ttr{ab}-\ttr{ab}
\equiv  0,
\]
and also
\[
\ttr{(ab(ba)b^2} \equiv  \ttr{ab^3}\ttr{ba} - \ttr{aba^{-1}b}
\equiv -\ttr{ab}\ttr{a^{-1}b} +\ttr{a^2}
\equiv \ttr{a^2}.
\]
Taking the difference we get
\[
\ttr{(ab^2)^2}-\ttr{(ab)^3} \equiv \ttr{a^2}.
\]
Now, the coset $\ttr{a^2}+\mathbb Z\ttr{ab}$ does not intersect the interval $\set{-1,0,1}$. Indeed, its points closest to zero are $\ttr{a^2}\ge2$ and $\ttr{a^2}-\ttr{ab}\le \ttr{a^2}-\ttr{b^2}-1\le -2$.
\end{proof}

The case $\ttr{(ab)^3}>\ttr{(ab^2)^2}$ of the dichotomy follows familiar patterns.

\begin{lemma}
Assume\label{61} $\ttr{(ab^2)^2}\le \ttr{(ab)^3}-2$.
Fix a word $w$, and let $k\geq 1$ be such that the length of
$wab^2(ab)^kab^2$ is a  multiple of $6$. If $w$ is empty, or is a product of $ab$ and $ab^2$, then we have
\[
\ttr{wab^2(ab)^kab^2}<\ttr{w(ab)^{k+3}}.
\]
\end{lemma}
\begin{proof}
On the left side we have
\begin{align*}
&\ttr{wab^2(ab)^kab^2}\\
&= \ttr{wab^2(ab)^k}\ttr{ab^2} - \ttr{wab^2(ab)^{k-1}ab^{-1}a^{-1}}\\
&=\ttr{w(ab)^k}\ttr{ab^2}^2 - \ttr{wb^{-1}(ab)^{k-1}}\ttr{ab^2} - \ttr{wab^2(ab)^{k-1}ab^{-1}a^{-1}}\\
&=\ttr{w(ab)^k}\bigl(\ttr{(ab^2)^2} + 2\bigr) - \ttr{wb^{-1}(ab)^{k-1}}\ttr{ab^2} - \ttr{wab^2(ab)^{k-1}ab^{-1}a^{-1}}\\
&\le \ttr{w(ab)^k}\ttr{(ab)^3} - \ttr{wb^{-1}(ab)^{k-1}}\ttr{ab^2} - \ttr{wab^2(ab)^{k-1}ab^{-1}a^{-1}}\\
&= \ttr{w(ab)^{k+3}} + \ttr{w(ab)^{k-3}} - \ttr{wb^{-1}(ab)^{k-1}}\ttr{ab^2} - \ttr{wab^2(ab)^{k-1}ab^{-1}a^{-1}}.
\end{align*}
We then have to prove that the sum of all but the first summands
of the above end result is negative. We will prove this fact by showing the following two inequalities:
\begin{align*}
\ttr{w(ab)^{k-3}} &\le \ttr{wb^{-1}(ab)^{k-1}}\ttr{ab^2},\\
\ttr{wab^2(ab)^{k-1}ab^{-1}a^{-1}}&>0.
\end{align*}

Suppose $k\geq 3$.
If $w$ is non empty, it begins with $ab$ and ends with $b$.
Thus all the inverted letters simplify and the inequalities follow from Lemma~\ref{ref13}(2).
This also holds	when $w$ is the empty word, with different simplifications.
	
Suppose $k<3$. Then $w$ cannot be empty, for reasons of length, and the second inequality causes no problems.
If $k=2$ the first inequality becomes $\ttr{wb^{-1}a^{-1}}\le \ttr{wb^{-1}ab}\ttr{ab^2}$, clearly true.
If $k=1$ it becomes
$\ttr{w(ab)^{-2}}\le \ttr{wb^{-1}}\ttr{ab^2}$, which is also true since $\ttr{w(ab)^{-2}}=\ttr{wb^{-1}a^{-1}}\ttr{ab}-\ttr{w}$.
\end{proof}

The other case $\ttr{(ab)^3}<\ttr{(ab^2)^2}$ will be treated via the following lemma.

\begin{lemma}
Assume\label{62} $\ttr{(ab)^3}\le \ttr{(ab^2)^2}-2$.
Fix a word $w$, and let $k,h\geq 0$ be such that the length of
$ab^2wab(ab^2)^kab(ab^2)^hab$ is a multiple of $6$. If $w$ is empty, or $w$ is a product of $ab$ and $ab^2$, then we have
\[
\ttr{ab^2wab(ab^2)^kab(ab^2)^hab} <
\ttr{ab^2w(ab^2)^{k+h+2}}.
\]
\end{lemma}

Unfortunately, although the end inequality is the same, the proof of Lemma~\ref{62} is harder than that of its twin Lemma~\ref{5C}.
This is due to the fact that the final step in the proof of Lemma~\ref{62}, namely the inequality
$\ttr{ab}^3\le\ttr{(ab^2)^2}$, may fail; for example, it fails for the matrices $A=L^{11}$, $B=LNL$ cited in~\S\ref{ref12}. We have thus to make do with the weaker $\ttr{ab}^3\le\ttr{(ab^2)^2}+\ttr{ab}$, that can still be quite narrow: for the above matrices we have
\[
\ttr{ab}^3=3375\le
3377 = 3362 + 15 =
\ttr{(ab^2)^2}+\ttr{ab}.
\]
The following lemma establishes that weaker inequality, as well as the fact that the minimal difference $\tr(AB)=15>14=\tr(B^2)$ for the above matrix pair is no coincidence.

\begin{lemma}
Assume\label{aiuto} $\ttr{(ab^2)^2}\geq \ttr{(ab)^3}+2$. Then the following formulas hold:
\begin{align}
\ttr{ab}&=\ttr{b^2}+1, \label{aiuto1}\\
2\ttr{a}&\le \ttr{b}, \label{aiuto2}\\
\ttr{ab}^3&\le \ttr{(ab^2)^2} +\ttr{ab}.\label{aiuto3}
\end{align}
\end{lemma}
\begin{proof}
Let $\alpha = \ttr{a}$, $\beta = \ttr{b}$, $x = \ttr{ab}$; then
we have $\ttr{(ab)^3}=T_3(x)=x^3-3x$ and $\ttr{(ab^2)^2}=T_2(\beta x-\alpha)=(\beta x-\alpha)^2-2$.

Assuming the negation of~\eqref{aiuto1}, we have
$x\geq\ttr{b^2}+2=\beta^2$ and thus
\begin{align*}
\ttr{(ab)^3}-\ttr{(ab^2)^2}
&= x^3 -\beta^2x^2 + (4\beta-3)x -2 \\
&\ge 	(4\beta-3)x -2.
\end{align*}
Since $\beta$ and $x$ are both at least $3$, the last term is positive, which is a contradiction;
this establishes~\eqref{aiuto1}.

Applying~\eqref{aiuto1} and its equivalent form $\ttr{b}^2=\ttr{ab}+1$ several times, we compute
\begin{align*}
\ttr{(ab^2)^2} &= \ttr{ab^2ab}\ttr{b} - \ttr{a^2b^2}\\
& = \ttr{abab}\ttr{b}^2 - \ttr{a^2b}\ttr{b} -\ttr{a^2b}\ttr{b} +\ttr{a^2}\\
& = \ttr{(ab)^2}\bigl(\ttr{ab}+1\bigr) -2\bigl(\ttr{a}\ttr{ab}\ttr{b}-\ttr{b}^2\bigr)+\ttr{a}^2-2\\
& = \ttr{(ab)^3} + \ttr{ab} + \ttr{(ab)^2} -2\ttr{a}\ttr{ab}\ttr{b} +2\ttr{b}^2+\ttr{a}^2-2\\
& = \ttr{(ab)^3} + \beta^2-1 + (\beta^2-1)^2-2 -2\alpha\beta(\beta^2-1) +2\beta^2 + \alpha^2-2\\
& = \ttr{(ab)^3} + \beta^4 -2\alpha\beta^3 +\beta^2 +2\alpha\beta +\alpha^2-4;
\end{align*}
therefore our hypothesis yield
\[
2\alpha\beta^3-\beta^4 \le
\beta^2 +2\alpha\beta +\alpha^2 -6.
\]
We change variables by setting $\alpha=\lambda\beta$, and obtain
\begin{align*}
(2\lambda -1)\beta^4 &\le (1+2\lambda +\lambda^2)\beta^2-6,\\
(2\lambda -1)\beta &< \frac{(1+\lambda)^2}{\beta},\\
2\lambda\beta &< \beta + \frac{(1+\lambda)^2}{\beta},\\
2\alpha &\le \beta + \biggl \lfloor \frac{(1+\lambda)^2}{\beta} \biggr \rfloor,
\end{align*}
the last step justified by the fact that $2\lambda\beta=2\alpha$ is an integer.
If $\beta=3$ then $\alpha=2$, $\lambda=2/3$, and the last inequality means $4\le 3$. Therefore $\beta \geq 4$ and the floor part is zero since $1+\lambda <2$; this proves~\eqref{aiuto2}.
	
We previously computed that
\[
\ttr{(ab^2)^2}-\ttr{(ab)^3} = \beta^4 -2\alpha\beta^3 +\beta^2 +2\alpha\beta +\alpha^2-4.
\]
Since $\ttr{ab}^3=\ttr{(ab)^3}+3\ttr{ab}$ and $\ttr{ab}=\beta^2-1$ we obtain
\begin{align*}
\ttr{ab}^3 &= \ttr{(ab^2)^2} -\beta^4 +2\alpha\beta^3 -\beta^2 -2\alpha\beta -\alpha^2+4 + 3\beta^2-3\\
&= \ttr{(ab^2)^2} -\beta^4 +2\alpha\beta^3 +2\beta^2 -2\alpha\beta -\alpha^2+1 \\
&= \ttr{(ab^2)^2} + (2\alpha-\beta)(\beta^3-\beta) +\beta^2 -\alpha^2+1\\
&\le \ttr{(ab^2)^2} +\beta^2 -\alpha^2+1\\
&\le \ttr{(ab^2)^2} +\ttr{ab},
\end{align*}    
thus settling~\eqref{aiuto3}.
\end{proof}

\begin{proof}[Proof of Lemma~\ref{62}]
Continuing from~\eqref{eq5} (whose proof does not depend on the relative values of $\ttr{a}$, $\ttr{b}$, $\ttr{sb}$) and applying~\eqref{aiuto3}, we obtain
\begin{align*}
\ttr{ab^2wab(ab^2)^kab(ab^2)^hab}
&= \ttr{w(ab^2)^{k+h+1}}\ttr{ab}^3
-\ttr{wb(ab^2)^{k+h}}\ttr{ab}^2\\
& \qquad - \ttr{wab(ab^2)^kb(ab^2)^{h}}\ttr{ab}
- \ttr{bwab(ab^2)^kab(ab^2)^h}\\
&\le \ttr{w(ab^2)^{k+h+1}}\ttr{(ab^2)^2}  \\
&\qquad + \ttr{w(ab^2)^{k+h+1}}\ttr{ab} - \ttr{wb(ab^2)^{k+h}}\ttr{ab}^2\\
& \qquad - \ttr{wab(ab^2)^kb(ab^2)^{h}}\ttr{ab} - \ttr{bwab(ab^2)^kab(ab^2)^h}\\
&= \ttr{w(ab^2)^{k+h+3}} + \ttr{w(ab^2)^{k+h-1}}\\
&\qquad + \ttr{w(ab^2)^{k+h+1}}\ttr{ab} - \ttr{wb(ab^2)^{k+h}}\ttr{ab}^2\\
& \qquad - \ttr{wab(ab^2)^kb(ab^2)^{h}}\ttr{ab} - \ttr{bwab(ab^2)^kab(ab^2)^h}.	
\end{align*}
Thus we need to prove that
\begin{equation}\label{666}
\begin{split}
\ttr{w(ab^2)^{k+h-1}} &+ \ttr{w(ab^2)^{k+h+1}}\ttr{ab}\\
< \ttr{wb(ab^2)^{k+h}}\ttr{ab}^2 + \ttr{wab(ab^2)^k&b(ab^2)^{h}}\ttr{ab} +  \ttr{bwab(ab^2)^kab(ab^2)^h}.
\end{split}
\end{equation}
We will need to switch the positions of some factors, and this will be accomplished by the formula
\begin{equation}\label{eq6}
\ttr{xyzw} = \ttr{xzyw} + \ttr{xz^{-1}yw} - \ttr{xyz^{-1}w},
\end{equation}
which moves $z$ to the left; of course, an analogous formula holds for moving to the right. Formula~\eqref{eq6} follows from Lemma~\ref{ref5}(2), since both $\ttr{xyzw} + \ttr{xyz^{-1}w}$ and
$\ttr{xzyw}+\ttr{xz^{-1}yw}$ equal $\ttr{xyw}\ttr{z}$.
	
In order to obtain~\eqref{666}, we work on the second summand of the second line. Moving $b$ to the left we get
\begin{align*}
\ttr{wab(ab^2)^kb(ab^2)^{h}}\ttr{ab}
&= \ttr{w(ab^2)^{k+h+1}}\ttr{ab} + \ttr{wa(ab^2)^{k+h}}\ttr{ab}\\
&\quad -\ttr{wab(ab^2)^{k-1}ab(ab^2)^{h}}\ttr{ab}\\
&= \ttr{w(ab^2)^{k+h+1}}\ttr{ab} + \ttr{wa(ab^2)^{k+h}}\ttr{ab} \\
&\qquad - \ttr{wab(ab^2)^{k-1}abab(ab^2)^{h}} - \ttr{wab(ab^2)^{k+h-1}}.
\end{align*}
Substituting this back into~\eqref{666}, the summand
$\ttr{w(ab^2)^{k+h+1}}\ttr{ab}$ simplifies.
The summand $\ttr{wab(ab^2)^{k-1}(ab)^2(ab^2)^{h}}$ is, by
Lemma~\ref{ref13}(2), always dominated by the last term
of~\eqref{666}, even when $k=h=0$.
Removing both of them we remain with the inequality
\begin{multline*}
\ttr{w(ab^2)^{k+h-1}} + \ttr{wab(ab^2)^{k+h-1}} \\
< \ttr{wb(ab^2)^{k+h}}\ttr{ab}\ttr{ab} +
\ttr{wa(ab^2)^{k+h}}\ttr{ab},
\end{multline*}
which holds by Lemma~\ref{ref13}(2). The case $k=h=0$ must be checked apart, but causes no problems.
\end{proof}

\begin{theorem}
Let\label{ref18} $A,B\in\SL_2\Zbb\p$, and assume $\tr(A)<\tr(B)$ and $\tr(AB)>\tr(B^2)$. If $\tr((AB)^3)>\tr((AB^2)^2)$,
then $ab$ is the only optimal word; otherwise, so is $ab^2$
\end{theorem}
\begin{proof}
As in the proof of Theorem~\ref{ref16}, let $u$ be a word of length a multiple of~$6$ which is trace-maximizing among all words of the same length. The statement will result by proving that $u$ is a power of $ab$ (in case $\ttr{(ab)^3}>\ttr{(ab^2)^2}$), or of $ab^2$ (in case $\ttr{(ab)^3}<\ttr{(ab^2)^2}$, no equality being possible by Lemma~\ref{dichotomy}).
	
Since $\ttr{b^2}<\ttr{ab}$, at least one $a$ appears in $u$, but $a^2$ does not by Lemma \ref{4-5-6}.
We claim that the factor $b^3$ is also excluded.
Indeed, for every $w$ we have
\begin{multline*}
\ttr{wab^3}=\ttr{wab}\ttr{b^2}-\ttr{wab\m}
\le\ttr{wab}\ttr{ab}-\ttr{wab}-\ttr{wab\m}\\
=\ttr{w(ab)^2}+\ttr{w}-\ttr{wa}\ttr{b}<\ttr{w(ab)^2}.
\end{multline*}
Therefore, $u$ factors uniquely as a product of the syllables
$ab$ and $ab^2$.

Suppose that both syllables appear; since $\abs{u}$ is a multiple
of~$6$, $ab^2$ must appear at least two times, and $ab$ at least three times. If $ab^2\prec ab$ then Lemma~\ref{61} contradicts the maximality of $u$, and the same does Lemma~\ref{62} if $ab\prec ab^2$. Thus only one syllable appears in $u$, and the proof is complete.
\end{proof}

\end{document}